\documentclass[reqno]{amsart}

\numberwithin{equation}{section} 
\usepackage{amsmath,amsfonts,amssymb,amsthm,latexsym}
\usepackage{enumerate}
\usepackage{cancel}
\usepackage{stackrel}

\usepackage{cases}
\usepackage[square,comma,numbers,sort&compress]{natbib}

\usepackage[colorlinks=true]{hyperref}
\hypersetup{urlcolor=blue, citecolor=red}

\newcounter{mnote}
\theoremstyle{plain}
\newtheorem{theorem}{Theorem}[section]
\newtheorem{proposition}[theorem]{Proposition}
\newtheorem{lemma}[theorem]{Lemma}

\theoremstyle{definition}

\theoremstyle{remark}
\newtheorem{remark}[theorem]{Remark}

\newcommand{\vect}[1]{\mathbf{#1}}

\newcommand{\bk}{\vect{k}}

\newcommand{\field}[1]{\mathbb{#1}}
\newcommand{\nN}{\field{N}}

\newcommand{\nZ}{\field{Z}}

\newcommand{\nR}{\field{R}}

\newcommand{\nH}{\field{H}}
\newcommand{\nV}{\field{V}}
\newcommand*  {\Z}{ {{{\mathbb Z}^*}^2}}

\begin{document}
\title[A Quasi-Geostrophic Two-layer  Model]{Long-time behavior of a two-layer model of baroclinic
quasi-geostrophic turbulence}

\date{April 14, 2012. {\it Journal of Mathematical Physics} {\bf (to appear)}.}

\author{Aseel Farhat}
\address[Aseel Farhat]{Department of Mathematics\\
                University of California\\
        Irvine, CA 92697-3875, USA}
\email[Aseel Farhat]{afarhat@math.uci.edu}
\author{R. Lee Panetta}
\address[R. Lee Panetta]{Department of Atmospheric Sciences\\
Texas AM University\\
Texas 77843, USA}
\email[R. Lee Panetta]{r-panetta@tamu.edu}
\author{Edriss S. Titi}
\address[Edriss S. Titi]{Department of Mathematics, and Department of Mechanical and Aero-space Engineering\\
University of California\\
Irvine, CA 92697-3875, USA.
Also, The Department of Computer Science and Applied Mathematics\\
The Weizmann Institute of Science, Rehovot 76100, Israel.
Fellow of the Center of Smart Interfaces (CSI), Technische Universit\"{a}t Darmstadt, Germany. }
\email[Edriss S. Titi]{etiti@math.uci.edu and edriss.titi@weizmann.ac.il}
\author{Mohammed Ziane}
\address[Mohammed Ziane]{Department of Mathematics\\
 University of Southern California\\
 Los Angeles, California 90089-1113, USA}
\email[Mohammed Ziane]{ziane@usc.edu}

\begin{abstract}
We study a viscous two-layer quasi-geostrophic beta-plane model that is forced by imposition of a spatially uniform vertical shear in the eastward (zonal) component of the layer flows, or equivalently a spatially uniform north-south temperature gradient.  We prove that the model is linearly unstable, but that non-linear solutions are bounded in time by a bound which is independent of the initial data and is determined only by the physical parameters of the model.  We further prove, using arguments first presented in the study of the Kuramoto-Sivashinsky equation, the existence of an absorbing ball in appropriate function spaces, and in fact the existence of a compact finite-dimensional attractor, and provide upper bounds for the fractal and Hausdorff dimensions of the attractor.  Finally, we show the existence of an inertial manifold for the dynamical system generated by the model's solution operator.  Our results provide rigorous justification for observations made by Panetta based on long-time numerical integrations of the model equations.
\end{abstract}
\maketitle

{\it Dedicated to Professor Peter Constantin on the occasion of his  60th birthday.}
 \bigskip
\section{Introduction}
The aim of this work is to present a mathematical analysis of a highly simplified, or ``minimal", quasi-geostropic model of baroclinic turbulence in a rapidly rotating atmosphere or ocean at large spatial scales.  The model is minimal in the sense that it has simple but geophysically motivated representations of central factors responsible for the large-scale features of extra-tropical flows, and little else. The number of layers is minimal for existence of a flow-determined temperature field, and the model includes a spatially featureless large-scale thermal forcing, a horizontal variation of the effective local rotation rate (through the beta-plane approximation), and forms of viscous dissipation at both large and small scales, without any imposition of flow structure by special features of the forcing or any  boundary walls.  This model is essentially the one introduced by \cite{Haidvogel_Held_1980} as a boundary-free modification of the original version of \cite{Phillips_2},  with the intent of creating a model of spatially ``homogeneous'' turbulence.  However numerical integrations \cite{Panetta-1993} showed the unsuspected development of spatial structures (zonal jets and associated ``storm tracks") of great persistence, and the presence of long time scales, and dynamically intrinsic spatial scales that are even now only understood by phenomenological arguments.  Our work is intended as a step toward explaining some of the features revealed by the numerical integrations.

Although the flow in each of the two layers is two-dimensional, the dynamical system as a whole has certain features in common with a one-dimensional system that also shows spatial pattern-forming features, the Kuramoto-Sivashinksy equation, as we explain later in this section. We then show in succeeding sections that arguments presented in \cite{Collet-etal}, \cite{Goodman} and \cite{Nicolaenko-Scheurer-Temam} for the case of the Kuramoto-Sivashinsky equation can be used to show the existence of a global attractor of finite fractal and Hausdorff dimension, as well as an inertial manifold.

We use the version of a  two-layer model that appeared in \cite{Haidvogel_Held_1980} and \cite{Panetta-1993}, which has two fluid layers of equal resting depth $H$, the lower layer has a slightly greater density, $\rho_2$, than the upper density, $\rho_1$. Motions are assumed to be quasi-geostrophic and to take place on a beta-plane, and the Bousinesq approximation is enforced (for a discussion of two-layer models see, e.g., \cite{Pedlosky-87}).  In terms of the dimensional streamfunctions $\Psi_i^*(x^*, y^*)$ in the layer $i=1,2$, the evolution equations for the potential vorticity fields are:
\begin{subequations}\label{two_layer_Panetta}
\begin{align}
\frac {\partial Q_1^*}{\partial t^*} +J(\Psi_1^*, Q^*_1)& = \kappa_T^*\hat\Psi^* +\nu^*(-\triangle)^3\Psi^*_1, \\
\frac {\partial Q_2^*}{\partial t^*} +J(\Psi_2^*, Q^*_2)& =-\kappa_T^*\hat\Psi^* - \kappa_M^*\triangle \Psi_2^*+\nu^*(-\triangle)^3\Psi^*_2,
\end{align}
where the potential vorticity $Q_i^*$ in the layer $i=1,2$ is given by:
\begin{equation}
 Q_i^*=f_0^* +\beta^*y^* +\Delta\Psi_i^* + \frac{(-1)^i}{\lambda^2}\hat\Psi^* ,
\end{equation}
\end{subequations}
where $\hat\Psi^*= (\Psi^*_1-\Psi^*_2)/2$, and $J(.,.)$ denotes the Jacobian:
\begin{equation*}
J(\psi,q) = \frac{\partial \psi}{\partial x}\frac{\partial q}{\partial y} - \frac{\partial \psi}{\partial y}\frac{\partial q}{\partial x}.
\end{equation*}
$\kappa_T^*$ is the temperature (or buoyancy) damping, $\kappa_M^*$  is the mechanical damping (i.e., surface drag or Ekman pumping), $\nu^*$ is small scale mixing diffusion constant (or the numerical viscosity), and $\lambda$ is the Rossby radius defined by $\lambda = \left[ g(\rho_2-\rho_1)H/(2\rho_2f_0)\right]^{1/2}$.

The deviations of solutions of system \eqref{two_layer_Panetta} from a time-invariant steady state with horizontally uniform vertical shear $U_0$ has been studied numerically in \cite{Haidvogel_Held_1980} and \cite{Panetta-1993}. Panetta assumed that the lower layer is at rest and the upper stream functions were written in the form:
\begin{align}
\Psi_1^*(t;x^*,y^*)=\psi_1^*(t;x^*,y^*)-U_0y^*, \quad \Psi_2^*(t;x^*,y^*)= \psi_2^*(t;x^*,y^*).
\end{align}
Here $\psi_1^*, \psi_2^*$ are the deviation stream functions and their corresponding potential vorticities are $q_1^*$ and $q_2^*$, respectively.

The transient stream-functions $\psi_i(t;x,y)$ (called ``eddy''), after
normalization of $\psi_i^*$, $q_i^*$, for $i=1,2$, $x^*$ and $y^*$ using velocity and length scales $U_0$ and $\lambda$ (see, e.g. \cite{Panetta-1993}) will satisfy the following evolution equations:
\begin{subequations}\label{two_layer_3}
\begin{align}
&\frac {\partial q_1}{\partial t} +J(\psi_1, q_1) = -\frac{\partial q_1}{\partial x} -(\beta +\frac 12)\frac{\partial \psi_1}{\partial x} +\kappa_T\hat\psi +\nu(-\Delta)^3\psi_1,  \label{model-1}\\
&\frac {\partial q_2}{\partial t} +J(\psi_2, q_2) = -(\beta -\frac 12)\frac{\partial \psi_2}{\partial x}-\kappa_M\Delta\psi_2 -\kappa_T\hat\psi +\nu(-\Delta)^3\psi_2, \label{model-2}\\
&\hat\psi= \frac 12(\psi_1-\psi_2), \quad q_1= \Delta \psi_1- \hat \psi, \quad q_2 = \Delta \psi_2+ \hat \psi. \label{N-q-psi}
\end{align}
\end{subequations}
Here, $\kappa_T, \kappa_M$ are the non-dimensional buoyancy and mechanical damping parameters, respectively, and can have any real value. $\beta$  is the non-dimensional central forcing parameter and can have any real value. The smaller $\beta$ is in absolute value, the stronger the forcing.
In the absence of any dissipation, non-dimensional beta must be less than $\frac 1 2$ in absolute value, otherwise, there is no instability. Addition of dissipative parameters changes the stability criterion a little.


The first two terms on the right-hand side of \eqref{model-1} and the first term on the right-hand side of \eqref{model-2} describe the interactions of the eddy fields with the imposed background state; the other terms on the right-hand side in each equation are the dissipation terms (whether physical or artificial/numerical). We note here that the order of the numerical dissipation $(-\Delta)^3$ is {\em ad hoc}: use of such ``hyperviscosity'' is common in geophysical models, where it functions as a computationally convenient method in spectral codes to simply absorb small scale energy and enstrophy, and to do so primarily over a small range of wave numbers at the high end of the spectrum in which there is no direct physical interest. Sometimes the hyperviscosity is applied to the $q_i$ rather than the $\psi_i$: from the phenomenological point of view, at high wavenumbers there is little difference between $q_i$ and $\Delta \psi_i$, and there is no physical basis to choose one form of dissipation rather than the other.

In geophysics, there are different types of multi-layer models. A derivation of a multi-layer model, in which the fluid consists of a finite number of homogeneous layers of uniform but distinct densities, is presented in \cite{Pedlosky-87}. As mentioned above, the two-layer model is the simplest layer model that has a representation of baroclinic (temperature-related) dynamics important in extratropical flows in planetary atmospheres and oceans.  The version treated in \cite{Pedlosky-87} differs from system \eqref{two_layer_3} in several of the linear non-conservative terms, the principal difference being the lack of numerical viscosity terms (i.e. $\nu = 0$).
The mathematical analysis of multi-layer models has been studied by several authors. For example Bernier in \cite{Ber} investigated the existence and uniqueness of solutions for a viscous multi-layer problem and proved the existence of absorbing sets and of a maximal attractor for the model; moreover, an upper bound on the dimension of the global attractor using the Sobolev-Lieb-Thirring lemma \cite{T} was obtained. Bernier and Chueshov, later in \cite{Chu}, investigated the finiteness of determining degrees of freedom for a similar multi-layer model. Their proof is based on the ideas introduced in \cite{Jones1}, \cite{Jones2} and \cite{Jones3} (see also \cite{Foias_Manley_Temam_Treve_1983} and \cite{Foias_Prodi_1967}). The models studied by Bernier and Chueshov are similar to the 2D Navier-Stokes equations and were supplemented with specific boundary conditions. We stress that our model is different from these models due to the fact that it is constructed about a background zonal shear flow, which makes the global nonlinear stability more involved.

The same two-layer quasi-geostrophic model \eqref{two_layer_3} was studied by Onica and Panetta in \cite{OP1} and \cite{OP2}. In \cite{OP1}, they obtained the global existence and uniqueness of weak solutions for $\psi = (\psi_1,\psi_2)$ in $H^1_{per}\times H^1_{per}$ in the case of dissipative terms $-(-\Delta)^{1+\alpha}q_1$, $-(-\Delta)^{1+\alpha} q_2$, respectively, where $\alpha$ is an arbitrary non-negative real number, and $\kappa_T =0$. Later on, in \cite {OP2}, they continued their work and proved the existence and uniqueness of classical solutions that are analytic in both space and time.

In the next section we will study the linear system and show that, for certain parameter values of the system \eqref{two_layer_3} (mainly when  the size of the domain is large enough),  the steady state solution $\psi_1 =\psi_2 = 0$ is linearly unstable. However, it has been  conjectured in \cite{Panetta-1993}, based on numerical evidence, that system \eqref{two_layer_3} admits a finite-dimensional global attractor, in particular, that the system is non-linearly globally stable, i.e.  all solutions are eventually (asymptotically in time) bounded. This observation  is an important tool in the physical and theoretical/analytical interpretations of the numerical results obtained in \cite{Panetta-1993}. As we will show in section 4, there is no issue obtaining the global existence and uniqueness of a weak solution for $q=(q_1,q_2) \in L^2_{per}\times L^2_{per}$, but the main challenge is to prove  Panetta's conjecture; i.e. proving the existence of a finite dimensional global attractor of system \eqref{two_layer_3}. The extra term $\frac{\partial q_1}{\partial x}$ that appears in the first equation of \eqref{two_layer_3}, which is due to the fact that the model is derived around a zonal shear flow, adds some difficulties in studying the global nonlinear stability of the system.

The numerical results obtained in  \cite{Panetta-1993} suggest a stabilization effect by the nonlinear terms. Remarkably, we recall that this is a well known phenomenon for certain nonlinear dissipative evolution equations such as the one-dimensional Kuramoto--Sivashinsky equation
$$
\frac{\partial w}{\partial t}+ \frac{\partial^4w}{\partial x^4}+
\frac{\partial^2w}{\partial x^2} + w\frac{\partial w}{\partial x}=0,
$$
and the one-dimensional Burgers--Sivashinsky equation
$$
\frac{\partial v}{\partial t}- \frac{\partial^2v}{\partial x^2}-
 v + v\frac{\partial v}{\partial x}=0,
$$
subject to periodic boundary conditions, with period $L>0$, in the basic period interval $\left[-\frac{L}{2},\frac{L}{2}\right]$, see for instance \cite{Collet-etal}, \cite{Goodman}, \cite{Nicolaenko-Scheurer-Temam}, \cite{T} and the references therein.

The Kuramoto--Sivashinksy equation has been studied extensively by many authors. In particular, the authors of \cite{Nicolaenko-Scheurer-Temam} proved the global  nonlinear stability, i.e. that the system is dissipative, in the case of odd-periodic solutions. This result was generalized to include all periodic solutions in \cite{Collet-etal} and \cite{Goodman}. See also the recent studies by Bronski and Gambill in \cite{Bronski_Gambill_2006} and by Giacomelli and Otto in \cite{Otto} and by Otto in \cite{Otto2}.

From the physical point of view, the nonlinear stability of the Kuramoto- Sivashinsky equation (KSE) can be explained as follows. The linear part of the KSE is unstable for the low wave numbers: the low modes grow exponentially, while the high modes decay exponentially. The KSE without the linear part is the inviscid Burgers equation which preserves the energy of the solution for as long as the solution exists and is smooth. On the other hand, we know that Burgers equation forms a shock in finite time, which is manifested by the blow-up of the derivative of the solution. That is, there is a mechanism of transferring energy from the low wave numbers to the high wave numbers. Putting all this together, we see that the nonlinearity in the one-dimensional KSE pushes energy from the low wave numbers to the high ones, hence stabilizes the system, because the linear part is strongly stable for large wave numbers.

From the mathematical point of view, the idea in the proof of the nonlinear stability in \cite{Nicolaenko-Scheurer-Temam}  for the case of an odd periodic solution relies on a decomposition of the solution $w(t;x)=\tilde w(t;x) +\varphi ,$ where $(1-\varphi _x)$ is a smooth approximation of the periodic delta function centered around $x = \pm \frac{L}{2}$. The idea of translation was later generalized in \cite{Collet-etal} and in \cite{Goodman} for general periodic solutions. In space dimension $n\geq 2$, the problem of global existence of regular solutions and their stability remains open; the methods employed in the one-dimensional case do not seem to be capable of extension to higher dimensional Kuramoto--Sivashinsky equations (for further discussion of this issue see, e.g., \cite{Bellout_Benachour_Titi_2003} and \cite{Cao_Titi_2006}). But, in the case of the two-dimensional Burger--Sivashinsky  equation, {the global existence of regular solutions follows from} the maximum principle.  We should mention, however, that for a variant of a 2D  Kuramoto--Sivashinsky equation, used for the study of waves in fluids on an inclined plane, the author of \cite{Pinto-1999} was able to generalize the methods developed by \cite{Collet-etal}, \cite{Goodman} and \cite{Nicolaenko-Scheurer-Temam}, and prove that the solutions to this two-dimensional model are uniformly asymptotically bounded, i.e. the system has an absorbing set. In fact, this system has a finite-dimensional global attractor.

Adapting similar tools and techniques developed for the one-dimensional Kuramoto-Sivashinsky equation, we prove the existence of an absorbing ball and a finite-dimensional global attractor for a slightly general version of system \eqref{two_layer_3}; that is when the dissipation terms are replaced by $(-\Delta)^{m}\psi_1$ and $(-\Delta)^{m}\psi_2$, respectively, when $m>5/2$. The condition $m>5/2$ is needed in our analysis in section 5 to prove the uniform boundedness in time of the kinetic energy of the system. We also prove the existence of an inertial manifold of the system and obtain an upper bound on the {dimension} of its global attractor.

We are interested in spatial periodic perturbations about the base background flow with periods $L$ in both directions. Therefore, we supplement system \eqref{two_layer_3} with periodic boundary condition over the domain $\Omega=\left[-\frac{L}{2}, \frac{L}{2}\right]^2,$ $(L\geq 1),$ and initial data $q_i(x, 0)= q_{0i}, \ i=1,2.$ We note that, due to the periodic boundary conditions, integrating \eqref{model-1} and \eqref{model-2} over $\Omega = \left[-\frac{L}{2}, \frac{L}{2}\right]$ yield
\begin{equation*}
\frac{d}{dt}\int_{\Omega}q_1\, dxdy= \frac{\kappa_T}2 \int_\Omega (\psi_1-\psi_2)\, dxdy,\quad
\frac{d}{dt}\int_{\Omega}q_2\, dxdy= -\frac{\kappa_T}2 \int_\Omega (\psi_1-\psi_2)\, dxdy.
\end{equation*}
Furthermore, by integrating \eqref{N-q-psi}, we also have
\begin{equation*}
\int_\Omega q_1= -\frac 12 \int_\Omega (\psi_1-\psi_2)\, dxdy, \quad \quad
\int_\Omega q_2= \frac 12 \int_\Omega (\psi_1-\psi_2)\, dxdy,
\end{equation*}
Therefore, if we assume $\int_{\Omega}q_{i}^0\, dxdy=0$, then $\int_{\Omega}q_{i}(x,t)\, dxdy=0$ for all $t\geq 0$, for $i=1,2$.
Furthermore, the spatial average of $\psi_1$ is equal to the spatial average of $\psi_2.$ Without loss of generality, we may assume that the averages over $\Omega$ of both $\psi_1$ and $\psi_2$ are zero. Moreover, we note that the system of equations:
\begin{align}
q_1= \Delta \psi_1- \frac1{2} (\psi_1-\psi_2), \quad q_2= \Delta \psi_2+ \frac1{2} (\psi_1-\psi_2),
\end{align}
where the unknowns are $\psi_1$ and $\psi_2$ is elliptic. Specifically, if we add and subtract the above equations, we obtain
\begin{align}\label{elliptic_two_layer}
q_1+q_2 = \Delta(\psi_1+\psi_2), \quad q_1-q_2 = \Delta(\psi_1-\psi_2) - (\psi_1-\psi_2).
\end{align}
Its clear now that system \eqref{elliptic_two_layer} is elliptic and thus, for given $q_1$ and $q_2$, we can solve for $(\psi_1-\psi_2)$ and $(\psi_1+\psi_2)$, under the imposed boundary condition and zero averages, uniquely. Thus we can solve system \eqref{elliptic_two_layer} uniquely for $\psi_1$ and $\psi_2$. {The solutions can be computed explicitly  in terms of the Fourier series.}

For simplicity of the presentation, in sections 5, 6, and 7, we will restrict ourselves to the case of  initial data that are odd with respect to the variable $y$. The uniqueness of local  solutions shows that the solution will remain odd in the $y$-direction for all time. Therefore, our solutions will satisfy the following boundary condition, which is a consequence of the odd symmetry:
\begin{equation}
q_i\left(x, -\frac{L}{2}\right)=q_i\left(x,\frac{L}{2}\right)=0,\quad \psi_i\left(x, -\frac{L}{2}\right)=\psi_i\left(x, \frac{L}{2}\right)=0,
\label{zero-condition}
\end{equation}
for all $ x\in \left[-\frac{L}{2}, \frac{L}{2}\right]$ and $i=1,2.$

The paper is organized as follows. In section 2, we study the linearized model, about the {steady} state $\psi_1=\psi_2 = 0$, and show that for certain range of the parameters this linearized model is unstable. In section 3, we introduce our notation and function spaces, as well as {some  useful} regularity estimates that will be used later to establish our results. In section 4, we prove the global existence of solutions. In section 5, we prove the existence of an absorbing ball as well as the existence of the global attractor. In section 6, we obtain an upper bound on the fractal and Hausdorff dimensions of the global attractor, and we prove the existence of an inertial manifold in section 7. Section 8 is an appendix in which we prove a Lieb-Thirring type inequality. This inequality provides a useful tool in obtaining an upper bound for the dimension of the global attractor of system \eqref{two_layer_3}.

\bigskip
\section{The linear problem}
In this section, we study the stability of the linearized version of system \eqref{two_layer_3} about the trivial state $\psi_1=\psi_2=0.$ We show that the linearized system is unstable for certain ranges of the parameters, which is an indication of a nontrivial dynamics of the full nonlinear system \eqref{two_layer_3}. The linearized system is given by:
\begin{subequations}\label{linear_two_layer}
\begin{align}
&\frac {\partial q_1}{\partial t} =-\frac{\partial q_1}{\partial x}-(\beta +\frac 12)\frac{\partial \psi_1}{\partial x} +\kappa_T\hat\psi+\nu(-\Delta)^3\psi_1  \label{linear-1}\\
&\frac {\partial q_2}{\partial t}  = -(\beta -\frac 12)\frac{\partial \psi_2}{\partial x}-\kappa_M\Delta\psi_2 -\kappa_T\hat\psi +\nu(-\Delta)^3\psi_2, \label{linear-2}\\
&\hat\psi = (\psi_1-\psi_2)/{2}, \quad q_1= \Delta\psi_1- \hat \psi, \quad q_2= \Delta\psi_2+ \hat \psi,\label{q-psi}\\
&\int_\Omega q_i\,dxdy=\int_\Omega \psi_i\,dxdy=0,\quad i=1,2.
\end{align}
\end{subequations}
We write the Fourier expansions of $q_j$ and $ \psi_j, \  j=1,2,$
\begin{equation*}
q_j(x) = \sum_{\bk\in {\Z}} q_{j\bk} \mbox{e}^{2\pi i\frac {x\cdot \bk}{L}}, \quad \quad \psi_j(x) = \sum_{\bk\in \Z} \psi_{j\bk} \mbox{e}^{2\pi i\frac {x\cdot \bk}{L}}, \quad j=1,2,
\end{equation*}
where $\Z= {\mathbb Z}^2-(0,0).$
From \eqref{q-psi}, we obtain
\begin{equation*}
q_{1\bk}= -\left(\frac {2\pi |\bk|}{L}\right)^2\psi_{1\bk} - \frac1{2} (\psi_{1\bk}-\psi_{2\bk}), \quad q_{2\bk}= -\left(\frac {2\pi |\bk|}{L}\right)^2\psi_{2\bk} + \frac1{2} (\psi_{1\bk}-\psi_{2\bk}),
\end{equation*}
for every $\bk \in \Z.$ Solving for $\psi_{1\bk}$ and $\psi_{2\bk}$, we have for $|\bk|\neq 0,$
\begin{align}
\psi_{1\bk}= -\alpha_{\bk}q_{1\bk}-\gamma_{\bk} q_{2\bk},\qquad \psi_{2\bk}= -\gamma_{\bk}q_{1\bk}-\alpha_{\bk} q_{2\bk},
\end{align}
where
\begin{equation}
\alpha_{\bk}=  \frac{\left(\frac {2\pi |\bk|}{L}\right)^2 +1/2}{\left(\frac {2\pi |\bk|}{L}\right)^4+
\left(\frac {2\pi |\bk|}{L}\right)^2},  \quad \text{and}\quad
\gamma_{\bk}=  \frac{1/2}{\left(\frac {2\pi |\bk|}{L}\right)^4+ \left(\frac {2\pi |\bk|}{L}\right)^2}.
\end{equation}

From the above one may conclude that there exists a constant $c_0(L,s)$ such that
\begin{equation}
\sum_{\bk\in \Z}|\bk|^{2+2s}(|\psi_{1\bk}|^2 +|\psi_{2\bk}|^2)\leq c_0(L,s)
\sum_{\bk\in \Z}|\bk|^{2s}(|q_{1\bk}|^2+|q_{2\bk}|^2),
\end{equation}
for all $s\in \mathbb R.$
It can be shown easily that system (\ref{linear-1}) and (\ref{linear-2}) is
equivalent to the infinite system of ordinary differential equations:
\begin{align}
\frac{dq_{1\bk}}{dt}= a_{\bk} q_{1\bk} +b_k q_{2\bk}, \qquad  \frac{dq_{2\bk}}{dt}= c_{\bk} q_{1\bk} +d_k q_{2\bk},
\end{align}
for all $\bk=(k_1,k_2) \in \Z$, where
\begin{align*}
a_{\bk}&=-\frac{2\pi ik_1}{L}\left(1- (\beta+1/2)\alpha_{\bk}\right) -\frac{\kappa_T}{2}(\alpha_{\bk}-\gamma_{\bk}) - \nu \left(\frac {2\pi |\bk|}{L}\right)^6\alpha_{\bk},\\
b_{\bk}&= \frac{2\pi ik_1}{L}(\beta+1/2)\gamma_{\bk}+ \frac{\kappa_T}{2}(\alpha_{\bk}-\gamma_{\bk}) - \nu\left(\frac {2\pi |\bk|}{L}\right)^6 \gamma_{\bk},\\
c_{\bk}&= \frac{2\pi ik_1}{L}(\beta-1/2)\gamma_{\bk}+\frac{\kappa_T}{2}(\alpha_{\bk}-\gamma_{\bk})- \kappa_M \left(\frac {2\pi |\bk|}{L}\right)^2\gamma_{\bk} -\nu\left(\frac {2\pi |\bk|}{L}\right)^6 \gamma_{\bk},
\end{align*}
\begin{align*}
d_{\bk}&= \frac{2\pi ik_1}{L}(\beta-1/2)\alpha_{\bk}-\frac{\kappa_T}{2}(\alpha_{\bk}-\gamma_{\bk}) -\kappa_M\left(\frac {2\pi |\bk|}{L}\right)^2\alpha_{\bk} -\nu \left(\frac {2\pi |\bk|}{L}\right)^6 \alpha_{\bk}.
\end{align*}
For $L\geq 1$, we note that for $|\bk|\geq 1$, we have
\begin{equation}
\alpha_{\bk}+\gamma_{\bk} = \left(\frac{L}{2\pi|\bk|}\right)^2, \quad \text{ and } \quad
|\alpha_{\bk}-\gamma_{\bk}| \leq 1.
\end{equation}
{Let $m_0\in \nN$, be fixed,} and assume that $|\bk|<m_0$, hence, we have
\begin{align*}
\left|\frac{\kappa_T}{2}(\alpha_{\bk}-\gamma_{\bk})\right|+ \nu\left(\frac {2\pi |\bk|}{L}\right)^6 (\alpha_{\bk}+\gamma_{\bk})\leq \frac{\kappa_T}{2} +\nu (2\pi m_0)^4,
\end{align*}
and
\begin{align*}
\kappa_M\left(\frac {2\pi |\bk|}{L}\right)^2(\alpha_{\bk}+\gamma_{\bk})+|\frac{\kappa_T}{2}(\alpha_{\bk}-\gamma_{\bk})|+\nu\left(\frac {2\pi |\bk|}{L}\right)^6 (\alpha_{\bk}+\gamma_{\bk})\\ \leq  \kappa_M+\kappa_T + \nu (2\pi m_0)^4.
\end{align*}

Note that the bounds above are independent of $L\geq 1$. Hence,
for any fixed $\epsilon>0$, and for $m_0$ fixed, there exists a $\delta(m_0)>0$, independent of $L$, such that if $\kappa_M+\kappa_T+\nu \leq \delta(m_0),$ then the real parts of $a_{\bk}, \ b_{\bk}, \ c_{\bk}$ and $d_{\bk}$ can be made smaller than $\frac\epsilon 2.$
Therefore, if we consider the matrix
\begin{equation}
M_{\bk}=
\begin{bmatrix}
a_{\bk}& b_{\bk}\\
c_{\bk} & d_{\bk}
\end{bmatrix},
\end{equation}
simple calculations will show that
\begin{align}
\mbox{tr}\,(M_{\bk})&= a_{\bk}+ d_{\bk} = -\frac{2\pi ik_1}L\left(1-2\beta \alpha_{\bk}\right) +\eta_1(\epsilon),\\
\mbox{det}\,(M_{\bk})&= \mbox{Re}\,(a_{\bk}d_{\bk} - c_{\bk}b_{\bk}) \notag \\
&=   \left(\frac {2\pi k_1}{L}\right)^2 \left[(\beta- 1/2)\alpha_{\bk} + (\alpha_{\bk}^2+\gamma_{\bk}^2)(\beta^2-1/4)\right]+\eta_2(\epsilon),
\end{align}
with $|\eta_i(\epsilon)|\leq \epsilon, \ i=1,2$, $\eta_i$ are independent of $L\geq 1$,
whenever $\kappa_M+\kappa_T+ \nu \leq \delta(k_0).$

Recall that the eigenvalues of $M_{\bk}$
are given by $1/2(\mbox{tr}\, (M_{\bk}) )\overline +\sqrt{\Delta_{\bk}}),$
where
\begin{equation}
\begin{aligned}
\Delta_{\bk}&= \mbox{Re}\,((\mbox{tr}\, (M_{\bk}))^2- 4 \mbox{det}\, (M_{\bk})).
\end{aligned}
\end{equation}
One can check that
\begin{align*}
\Delta_{\bk} &= -\left(\frac {2\pi k_1}{L}\right)^2 \left((1-2\beta\alpha_k)^2 + 4\left(\alpha_{\bk}(\beta-1/2) + (\gamma_{\bk}^2-\alpha_{\bk}^2)(\beta^4-1/4)\right)\right) + \eta_3(\epsilon)\\
&= \left(\frac{2\pi k_1}{L}\right)^2 \beta_{\bk}^2 \left( (1-4\beta^2) - \left(\frac{1-\alpha_{\bk}}{\beta_{\bk}}\right)^2 \right) + \eta_3(\epsilon),
\end{align*}
where $|\eta_3(\epsilon)|<\epsilon.$ Since $\frac{1-\alpha_{\bk}}{\beta_{\bk}}= 2\left(\frac12- \left(\frac{2\pi|k|}{L}\right)^4\right)$, then
\begin{align}
\Delta_{\bk} & = \left(\frac{2\pi k_1}{L}\right)^2 \beta_{\bk}^2 \left( (1-4\beta^2) - 4\left(\frac12- \left(\frac{2\pi|k|}{L}\right)^4\right)^2\right) + \eta_3(\epsilon).
\end{align}
Now, we observe that if we choose $|\beta| \ll1/4$, and choose $L\geq1$, large enough, such that
\begin{align}
\left(\frac{2\pi}{L}\right)^4 \sim \frac 3 8 \qquad \text{and} \qquad \left(\frac{2\pi k_0}{L}\right)^4 \sim \frac 1 2,
\end{align}
then, one can easily see that, under these conditions, $\Delta_{\bk}$ is non-negative. Therefore, one of the eigenvalues of $M_{\bk}$ will have a positive real part. This proves the following proposition.

\begin{proposition}
The linear system \eqref{linear_two_layer} is unstable for $\kappa_M, \kappa_T,$ and $\nu$ small, and $L$ large enough.
\end{proposition}

\bigskip
\section{Function spaces and notation}
Let $\Omega=\left[-\frac{L}{2},\frac{L}{2}\right]^2$. We denote by ${\dot L}^2_{per}(\Omega)$ the Hilbert space of $\Omega-$periodic functions $f$ defined on ${\mathbb R}^2$ such that
$$
f|_\Omega\in L^2_{per}(\Omega) \quad \mbox{and } \ \int_{\Omega} f\,dxdy =0.
$$
Let $\|\cdot\|_{L^2}, \ (\cdot,\cdot)$ be the norm and inner product, respectively, in $L^2_{per}(\Omega).$ We also denote by ${\dot H}^s_{per}(\Omega), \ s\in \mathbb R$ the space of $\Omega -$periodic functions with zero mean characterized by their Fourier expansion:
\begin{align*}
&{\dot H}^s_{per}(\Omega)=\notag \\
&\left\{ u: u= \sum_{\bk\in {\Z}}u_{\bk} e^{2\pi i \frac{\bk\cdot {\bf x}}{L}}, \bar u_{\bk} =u_{-\bk}, \; \sum_{\bk\in {\Z}}\left(\frac {2\pi |\bk|}{L}\right)^{2s} |u_{\bk}|^2<\infty, \ u_{\bf 0}=0\right\}.
\end{align*}
We denote $A=-\Delta$, with domain $D(A)={\dot H}^2_{per}(\Omega).$ When $D(A)$ is considered as a subset of ${\dot L}^2_{per}(\Omega)$ with the $L^2$ topology, the operator $A: {\dot L}^2_{per}(\Omega)\to {\dot L}^2_{per}(\Omega)$ is an unbounded, self-adjoint, positive operator, with compact inverse $A^{-1}.$ For
$$
u({\bf x})= \sum_{\bk\in \Z}u_{ k}e^{2\pi i \frac{\bk\cdot  {\bf x}}{L}} \quad \in{\dot H}^2_{per}(\Omega) \cap {\dot H}^s_{per}(\Omega) ,
$$
we have
\begin{equation}
Au({\bf x})= \sum_{\bk\in \Z}\left(\frac{2\pi|\bk|}{L}\right)^2u_{\bk}e^{2\pi i \frac{\bk\cdot  {\bf x}}{L}}, \quad A^{\frac s2}u({\bf x})=\sum_{\bk\in \Z}\left(\frac{2\pi|\bk|}{L}\right)^{s}u_{\bk}e^{2\pi i \frac{\bk\cdot {\bf x}}{L}},
\end{equation}
for any $s\in \mathbb R$. The space ${\dot H}^s_{per}(\Omega)$ is a Hilbert space with the  norm
\begin{equation}
||u||_{H^s}^2:= \|A^{\frac s2}u\|_{L^2}^2= \sum_{\bk\in \Z}\left(\frac{2\pi|\bk|}{L}\right)^{2s} |u_{\bk}|^2< \infty.
\end{equation}
We also denote $\nH={\dot L}^2_{per}\times {\dot L}^2_{per}$ and $\nV^s={\dot H}^s_{per}\times {\dot H}^s_{per}$. We denote $A_y=-\frac{\partial^2}{\partial y^2}$ and its domain ${H}^s_{per}(dy)(\Omega)$ the space of functions that are characterized by:
\begin{align}
&{H}^s_{per}(dy)(\Omega)=\notag \\
&\left\{ u: u= \sum_{k\in {\nZ}}u_{k}(x)e^{2\pi i \frac{ky}{L}}, \; \sum_{k\in {\nZ}}\left(\frac {2\pi k}{L}\right)^{2s} |u_{k}(x)|^2<\infty, \text{for all }x\in\left[-\frac L2, \frac L2\right]\right\}.
\end{align}
For $u \in {H}^s_{per}(dy)$, we have
\begin{align}
A_y^{\frac s2}u({\bf x})=\sum_{k\in \nZ}\left(\frac{2\pi k}{L}\right)^{s}u_{k}(x)e^{2\pi i \frac{k{y}}{L}}.
\end{align}
We also define
\begin{align}
\|f(x,.)\|_{L^2_y}^2:= \int_{-\frac L2}^{\frac L2}f^2(x,y)\, dy,
\end{align}
and
\begin{align}
{\dot H}^s_y(\Omega) := \left\{ u: u \in H^s_{per}(dy), \int_{-\frac L2}^{\frac L2} u(x,y)\,dy = 0, \text{ for all } x\in\left[-\frac L2,\frac L2\right]\right\}.
\end{align}
This space is a Hilbert space with the norm $||u(x,.)||_{{\dot H}^s_y}^2:=  \|A_y^{\frac{s}{2}}u(x,.)\|^2_{L^2_y}$.
\begin{remark}
In this paper, $C$ represents a positive dimensionless scale-invariant constant that may change from line to line. $C_i$ represents a positive constant that may depend on the parameters $\beta$, $\kappa_M$, $\kappa_T$, $\nu$, $m$, and $L$.
\end{remark}
Let $s>0$, and assume that $(Q_1,Q_2)\in \nV^{s-2}$. Let $(\Psi_1, \Psi_2)$ be the unique solution of the elliptic system
\begin{align}
&Q_1=-A\Psi_1 -\frac 12(\Psi_1-\Psi_2),\qquad Q_2=-A\Psi_2 +\frac 12(\Psi_1-\Psi_2),\\
&\int_{\Omega}\Psi_i\, dxdy=0, \quad i=1,\, 2.
\end{align}
By elliptic regularity, one has $(\Psi_1, \Psi_2) \in \nV^s.$ Moreover, one can easily prove the existence of a scale-invariant positive constant $C$ such that, for  $i=1,2$, and $s\geq r\geq 0$,
\begin{align}
\|\Psi_i\|_{L^2} \leq CL \|A^{1/2} \Psi_i\|_{L^2} \leq CL^2 \|A \Psi_i\|_{L^2},\\
\quad \|A^r\Psi_i\|_{L^2} \leq CL^{2(s-r)}\|A^s\Psi_i\|_{L^2}, \label{Poincare}
\end{align}
and
\begin{align}
\frac{1}{C}\left(\|A \Psi_1\|_{L^2}^2+\|A \Psi_2\|_{L^2}^2\right)&\leq \|Q_1\|_{L^2}^2+\|Q_2\|_{L^2}^2 \notag \\
&\leq C(1+L^4)\left(\|A \Psi_1\|_{L^2}^2+\|A \Psi_2\|_{L^2}^2\right).\label{regularity}
\end{align}
We recall the following Sobolev interpolation inequality. For $\phi \in H^3_{per}(\Omega)$, there exists a positive scale-invariant constant $C$ such that
\begin{align}\label{Sobolev_interpolation}
\|A\phi\|_{L^2} \leq C \|A^{1/2}\phi\|_{L^2}\|A^{3/2}\phi\|_{L^2}.
\end{align}
We also recall the following Agmon's inequality in space dimension two (see, e.g., \cite{CF88}). For $\phi \in H^2_{per}(\Omega)$, there exists an absolute positive constant $C$ such that
\begin{equation}
\|\phi\|_{L^{\infty}} \leq C \|\phi\|_{L^2}^{1/2}\|\phi\|^{1/2}_{H^2}. \label{Agmon}
\end{equation}
We now define the trilinear form $b(\cdot , \cdot ,\cdot )$ by
\begin{equation}
b(\Psi, Q, \bar Q)=\int_{\Omega}J(\Psi,Q)\bar Q\, dxdy= \int_{\Omega}\left(\frac{\partial\Psi}{\partial x} \frac{\partial Q}{\partial y}-\frac{\partial\Psi}{\partial y}\frac{\partial Q}{\partial x}\right)\bar Q\,dxdy,
\end{equation}
whenever the integral makes sense. We have in particular the following property
\begin{equation}
b(\Psi, Q,Q)=0 \quad \text{for every } \Psi \in {\dot H}^2_{\text{per}}(\Omega)\
\text{and } Q\in {\dot H}^1_{\text{per}}(\Omega).
\end{equation}

\bigskip
\section{Global existence of solutions}
In this section we consider the slightly more general system of nonlinear equations than system \eqref{two_layer_3}
\begin{subequations}\label{two_layer_m}
\begin{align}
&\frac {\partial q_1}{\partial t} +J(\psi_1, q_1) =-\frac{\partial q_1}{\partial x}-(\beta +\frac 12)\frac{\partial \psi_1}{\partial x} +\kappa_T\hat\psi+\nu A^m\psi_1  \label{model-1r}\\
&\frac {\partial q_2}{\partial t} +J(\psi_2, q_2) =-(\beta -\frac 12)\frac{\partial \psi_2}{\partial x}-\kappa_M\triangle\psi_2-\kappa_T\hat\psi+\nu A^m\psi_2, \label{model-2r}\\
& q_1 = -A\psi_1 - \hat\psi, \qquad q_2 = -A\psi_2+\hat \psi, \qquad \hat\psi = \frac{1}{2}(\psi_1-\psi_2),\\
&q_1(\cdot, 0)=q_1^0, \quad q_2(\cdot, 0)=q_2^0,
\end{align}
\end{subequations}
where $m> 5/2$, is a given positive real number, and $(q^0_1,q_2^0)$ is given in $\nH$. This choice of $m$ implies that the dissipation in the model is slightly stronger than $(-\Delta)^{3/2} q_i,\ i=1,2,$ plus lower order terms. This system is an obvious generalization of system \eqref{two_layer_3}, in the sense that it contains a more general dissipative terms, which may or may not be symmetric.

First, we shall say few words about the existence of solutions of the nonlinear system \eqref{two_layer_m}. Following ideas from the theory of the 2D Navier-Stokes equations, it is not difficult to show that system \eqref{two_layer_m} is well-posed for all times (see e.g., \cite{OP1} and \cite{OP2}). However, a {direct proof}, which is based on energy estimates for $\psi_i$ and $q_i$, $i=1, 2$, for arbitrary large intervals of time, does not yield the uniform boundedness of solutions in time, and does not preclude a possible exponential growth in time of solutions. A refinement of the proof in the next section will show that  solutions are actually uniformly bounded in time, and the dynamical system associated with the two-layer problem admits a compact global attractor.
\begin{remark}
The arguments and the estimates in this section are formal and they can be justified rigorously following the usual Galerkin approximation procedure and passing to the limit using the appropriate compactness theorem of the {Aubin-Lions type} (see, e.g., \cite{CF88}, \cite{Lions}, \cite{T}).
\end{remark}

{It is clear} that:
\begin{align}
(\kappa_T \hat\psi, \psi_1) - (\kappa_T \hat\psi, \psi_2)= 2\kappa_T\|\hat \psi\|_{L^2}^2, \text{   and   } (\kappa_M A\psi_2, \psi_2) = \kappa_M \|A^{1/2}\psi_2\|_{L^2}^2.
\end{align}
Taking the $L^2$ inner product of \eqref{model-1r} with $-\psi_1$ and \eqref{model-2r} with $-\psi_2$ and adding the equations we obtain:

\begin{align}\label{Ee1}
\frac12 \frac{d}{dt} \left(\|A^{1/2} \psi_1\|^2 + \|A^{1/2}\psi_2\|_{L^2}^2 + 2\|\hat\psi\|_{L^2}^2\right) + &\nu\sum_{i=1,2} \|A^{m/2}\psi_i\|_{L^2}^2 + 2\kappa_T\|\hat\psi\|_{L^2}^2  \notag \\
&+ \kappa_M\|A^{1/2}\psi_2\|_{L^2}^2 =(\frac{\partial q_1}{\partial x},\psi_1).
\end{align}
Since by integration by parts we have $(\frac{\partial q_1}{\partial x}, \psi_1) = - (\frac{\partial \psi_1}{\partial x}, q_1)$, by the virtue of estimates \eqref{Poincare} and \eqref{regularity} we have
\begin{align}\label{energy_existence_two_layer_m}
&\frac12 \frac{d}{dt} \left(\|A^{1/2} \psi_1\|_{L^2}^2 + \|A^{1/2}\psi_2\|_{L^2}^2 + 2\|\hat\psi\|_{L^2}^2\right) +\nu\sum_{i=1,2} \|A^{m/2}\psi_i\|_{L^2}^2 \notag \\
& \qquad \qquad \leq \|A^{1/2}\psi_1\|_{L^2}\|q_1\|_{L^2}\notag \\
& \qquad \qquad \leq C(1+L^4)^{1/2}L^{m-2}  \|A^{1/2}\psi_1\|_{L^2}\|A^{m/2}\psi_1\|_{L^2}\notag \\
& \qquad \qquad \leq \frac{\nu}{2} \|A^{m/2}\psi_1\|_{L^2}^2 + \frac{C(1+L^4)L^{2m-4}}{2\nu} \|A^{1/2}\psi_1\|_{L^2}^2,
\end{align}
which can be written as
\begin{align}
&\frac{d}{dt} \left(\|A^{1/2} \psi_1\|_{L^2}^2 + \|A^{1/2}\psi_2\|_{L^2}^2 + 2\|\hat\psi\|_{L^2}^2\right) +\nu\sum_{i=1,2} \|A^{m/2}\psi_i\|_{L^2}^2 \notag \\
& \qquad \qquad \qquad \leq C_1\left(\|A^{1/2} \psi_1\|_{L^2}^2 + \|A^{1/2}\psi_2\|_{L^2}^2 + 2\|\hat\psi\|_{L^2}^2\right),
\end{align}
where $C_1= C_1(\nu, m, L)= \frac{C(1+L^4)L^{2m-4}}{\nu}$.
Therefore, given initial data $\psi_i^0 = \psi_i(0), i = 1,2$, in ${\dot H}^1_{per}$, and $T>0$ fixed, the solutions of system \eqref{two_layer_m} will satisfy
\begin{align}
\psi_i \in L^{\infty}(0,T; {\dot H}^1_{per}) \cap L^2(0,T; {\dot H}^m_{per}), \quad i =1,2.
\end{align}

Now, we observe that:
\begin{align}
\kappa_T (\hat \psi,q_1) - \kappa_T(\hat \psi, q_2) = \kappa_T(\hat \psi, q_1-q_2)& = -2\kappa_T(\hat \psi, A \hat \psi +  \hat \psi)\notag \\
& = -2\kappa_T\|A^{1/2} \hat \psi\|_{L^2}^2 -2\kappa_T\|\hat \psi\|_{L^2}^2,
\end{align}
and that
\begin{align}
\kappa_M(A \psi_2, q_2)  = \kappa_M(A \psi_2, -A \psi_2+\hat \psi)& = -\kappa_M\|A \psi_2\|_{L^2}^2 + \kappa_M(A \psi_2,\hat \psi)\notag \\
& \leq -\frac{\kappa_M}{2}\|A \psi_2\|_{L^2}^2 + \frac{\kappa_M}{2}\|\hat \psi\|_{L^2}^2,
\end{align}
as well as that
\begin{align}
(A^m\psi_1,q_1)+(A^m\psi_2,q_2) &= (A^m \psi_1, -A\psi_1-\hat \psi) + (A^m \psi_2,-A\psi_2+ \hat \psi) \notag \\
& = -\sum_{i=1,2} \|A^{(m+1)/2} \psi_i\|_{L^2}^2 - 2\|A^{m/2}\hat \psi\|_{L^2}^2.
\end{align}
Multiplying \eqref{model-1r} by $q_1$ and \eqref{model-2r}  {by $q_2,$ integrating over $\Omega,$ and adding} the equations we get:
\begin{align}\label{E2}
&\frac1 2\frac{d}{dt}\left(\|q_1\|_{L^2}^2+\|q_2\|_{L^2}^2\right)+ \nu\sum_{i=1,2}\|A^{(m+1)/2} \psi_i\|_{L^2}^2 + 2\nu \|A^{m/2}\hat \psi\|_{L^2}^2 +2\kappa_T\|A^{1/2}\hat \psi\|_{L^2}^2 \notag \\
&\qquad  \leq -(\beta+\frac12)\left(\frac{\partial \psi_1}{\partial x}, q_1\right)- (\beta-\frac12)\left(\frac{\partial \psi_2}{\partial x},q_2\right)+  \left(\frac{\kappa_M}{2}-2\kappa_T\right)\|\hat \psi\|_{L^2}^2.
\end{align}

Adding the equation \eqref{Ee1} and \eqref{E2} implies that
\begin{align}\label{E3}
&\frac1 2 \frac{d}{dt}  \left(\sum_{i=1,2}\|q_i\|_{L^2}^2+ \sum_{i=1,2}\|A^{1/2}\psi_i\|_{L^2}^2 + 2\|\hat \psi\|_{L^2}^2\right) \notag \\
&\qquad +  \nu\left(\sum_{i=1,2} \left(\|A^{(m+1)/2} \psi_i\|_{L^2}^2+\|A^{m/2} \psi_i\|_{L^2}^2\right)+ 2\|A^{m/2}\hat \psi\|_{L^2}^2\right) \notag \\
&\qquad \leq -(\beta+1)\left(\frac{\partial \psi_1}{\partial x},q_1\right) - (\beta-1)\left(\frac{\partial \psi_2}{\partial x},q_2\right) + \left|\frac{\kappa_M}{2}-2\kappa_T\right|\|\hat \psi\|_{L^2}^2\notag \\
&\qquad \leq (|\beta|+1) \sum_{i=1,2}\|A^{1/2} \psi_i\|_{L^2}\|q_i\|_{L^2} +\left|\frac{\kappa_M}{2}-2\kappa_T\right|\|\hat \psi\|_{L^2}^2\notag \\
& \qquad \leq \frac{1}{2} \sum_{i=1,2} \left((|\beta|+1)^2\|A^{1/2} \psi_i\|_{L^2}^2 +\|q_i\|_{L^2}^2\right)  +\left|\frac{\kappa_M}{2}-2\kappa_T\right|\|\hat \psi\|_{L^2}^2.
\end{align}
Then,
\begin{align}\label{enstrophy_existence_two_layer_m}
&\frac{d}{dt} \left(\sum_{i=1,2}\|q_i\|_{L^2}^2+ \sum_{i=1,2}\|A^{1/2}\psi_i\|_{L^2}^2 + 2 \|\hat \psi\|_{L^2}^2\right)\notag \\
 &\qquad \qquad \qquad \leq C_2\left(\sum_{i=1,2}\|q_i\|_{L^2}^2+ \sum_{i=1,2}\|A^{1/2}\psi_i\|_{L^2}^2 + 2 \|\hat \psi\|_{L^2}^2\right).
\end{align}
where $C_2 = C_2(\beta, L, \kappa_M, \kappa_T)$.

Therefore, given the initial data $q^0_i=q_i(0), i=1,2$, in ${\dot L}^2_{\text{per}}$, and $T>0$ fixed, the solutions of system \eqref{two_layer_m}
satisfy
$$
q_i\in L^{\infty}(0,T; {\dot L}^2_{\text{per}})\cap L^{2}(0,T;{\dot H}^{m-1}_{\text{per}}), \; i=1,2.
$$
The estimates \eqref{energy_existence_two_layer_m} and \eqref{enstrophy_existence_two_layer_m} play an essential role in proving the global existence, uniqueness, and the continuous dependence on initial data of solutions (see, for instance, \cite{CF88}, \cite{OP1}, \cite{OP2}, and \cite{T}).
In particular, this allows us to define the semigroup of solution operators
$$
S(t):\ (q^0_1,q_2^0)\in \nH \mapsto (q_1(t),q_2(t))\in \nH.
$$
We summarize the above results in the following:
\begin{theorem}
Let $(\psi_1^0,\psi_2^0)\in \nV^1$ be given. Then there exists a unique solution $(\psi_1,\psi_2)$ of system \eqref{two_layer_m} such that
$$
(\psi_1,\psi_2) \in C([0,T];\nV^1)\cap L^2(0,T; \nV^{m}), \quad \text{for all  } T>0.
$$
\end{theorem}
\begin{remark}
We remark that the above result was established by Onica and Panetta in \cite{OP1}. They gave a rigorous proof for the well-posedness of system \eqref{two_layer_m} in the space $\psi = (\psi_1,\psi_2) \in H^1_{per}\times H^1_{per}$, when $\kappa_T=0$.
\end{remark}
\begin{theorem}
Let $(q^0_1,q_2^0)\in \nH$ be given. Then there exists a unique solution $(q_1,q_2)$ of system \eqref{two_layer_m} such that
$$
(q_1,q_2) \in C([0,T];\nH)\cap L^2(0,T; \nV^{m-1}), \quad \text{for all  }T>0,
$$
and the semigroup
$$
S(t):\ (q^0_1,q_2^0)\in \nH\mapsto (q_1(t),q_2(t))\in \nH.
$$
is continuous from $\nH$ into $\nV^{m},$ for all $t>0.$
\end{theorem}
\begin{remark}
Further regularity results were established in \cite{OP2}; the authors proved the Gevrey regularity of solutions  following the work of Foias and Temam \cite{FT} and its generalization to nonlinear analytic parabolic equations in \cite{FeTi}.
\end{remark}

\bigskip
\section{Absorbing Sets and Attractors}
In this section, we will follow the Nicolaenko, Scheurer, and Temam (N-S-T) idea that was used in \cite{Nicolaenko-Scheurer-Temam} for the one-dimensional Kuramoto-Sivashinsky equation in the case of odd periodic solutions, to improve the estimates introduced in the previous section. We notice that if the initial conditions $q_1^0$ and $q_2^0$ are periodic functions, odd in the variable $y$, then the initial conditions $\psi_1^0$ and $\psi_2^0$ are  periodic functions which are also odd in the variable $y$. Moreover, one can easily check that if $q_1(t;x,y)$ and $q_2(t;x,y)$ are solutions for system \eqref{two_layer_m}, then $-q_1(t;x,-y)$ and $-q_2(t;x,-y)$ also satisfy the equations with the same initial values, $-q_1^0(x,-y)= q_1^0(x,y)$ and $-q_2^0(x,-y) = q_0^1(x,y)$. By the uniqueness of the solutions of system \eqref{two_layer_m} we conclude that $q_1(t;x;y) = -q_1(t;x,-y)$ and $q_2(t;x,y)=-q_2(t;x,-y)$. Consequently, the solutions $q_1$ and $q_2$ are odd in the variable $y$. {A similar argument} will show that $\psi_1$ and $\psi_2$ are also odd in the variable $y$. So the space of functions that are {periodic and odd  in the variable $y$} is invariant under the solutions of system \eqref{two_layer_m}. Thus, we will restrict ourselves in this and the following sections to the case that $\psi= (\psi_1,\psi_2)$ and $q = (q_1, q_2)$ are periodic functions which are odd periodic in the variable $y$ in $\Omega$. In this case, we will prove the existence of an absorbing ball (i.e., the solutions remain uniformly bounded in time) as well as the existence of a global attractor for system \eqref{two_layer_m}.

Following the N-S-T idea in \cite{Nicolaenko-Scheurer-Temam}, let $\bar\psi(y)$ be an odd ${C}^\infty$ $L$-periodic function on $\Omega$ that depends only on the $y$ variable that will be determined later. We set
\begin{align}
\psi_1(t;x,y) = \bar\psi(y) + \Psi_1(t;x,y), \quad q_1(t;x,y) = \bar{q}(y) + Q_1(t;x,y),
\end{align}
where
\begin{align}
\bar q(y) = -A_y\bar \psi(y) - \frac {\bar \psi(y)}{2},\quad Q_1(t;x,y) = -A\Psi_1(t;x,y) - \hat \Psi(t;x,y), \\
\hat \Psi(t;x,y)= \frac{1}{2}(\Psi_1(t;x,y) - \psi_2(t;x,y)).
\end{align}
After we substitute back, the nonlinear system \eqref{two_layer_m} in terms of $\Psi_1, \psi_2, Q_1$, and $q_2$ will become

\begin{subequations}\label{two_layer_m_NST_Q}
\begin{align}
&\frac {\partial Q_1}{\partial t} + J(\Psi_1, Q_1) + J(\Psi_1,\bar q) + J(\bar \psi, Q_1)= -\frac{\partial Q_1}{\partial x} -(\beta +\frac 12)\frac{\partial \Psi_1}{\partial x} +\kappa_T\hat\Psi \notag \\
&\qquad \qquad \qquad \qquad \qquad \qquad \qquad \qquad \qquad \;+\nu A^m\Psi_1 +g_1,\label {eq21} \\
&\frac {\partial q_2}{\partial t} +J(\psi_2, q_2) = -(\beta -\frac 12)\frac{\partial \psi_2}{\partial x}+\kappa_MA\psi_2 -\kappa_T\hat\Psi +\nu A^m\psi_2 + g_2, \label{eq22}\\
&g_1 = \frac 1 2 \kappa_T \bar\psi + \nu A_y^m\bar\psi,\qquad g_2 = -\frac 12 \kappa_T \bar \psi, \quad \bar q= -A_y \bar \psi- \frac {\bar \psi}{2}, \\
&Q_1 = -A\Psi_1- \hat \Psi(t;x,y),\quad q_2 = -A\psi_2 + \hat \Psi + \frac{1}{2}\bar\psi,\quad \hat \Psi= \frac{1}{2}(\Psi_1 - \psi_2),\\
&Q_1(0,\cdot)=Q_1^0 = q_1^0 - {\bar q} , \quad q_2(0, \cdot) =q_2^0.
\end{align}
\end{subequations}
Multiplying \eqref{eq21} by $-\Psi_1$ and \eqref{eq22} {by $-\psi_2,$ integrating over $\Omega,$ and adding the equations, we obtain}
\begin{align}\label{E1}
&\frac 12 \frac{d}{dt}\left( \|A^{1/2} \Psi_1\|_{L^2}^2 +  \|A^{1/2} \psi_2\|_{L^2}^2 + 2\|\hat\Psi\|_{L^2}^2 \right) + \nu \left(\|A^{m/2} \Psi_1\|_{L^2}^2 + \|A^{m/2}\psi_2\|_{L^2}^2 \right) \notag \\
&+2\kappa_T \|\hat \Psi\|_{L^2}^2 + \kappa_M\|A\psi_2\|_{L^2}^2 = b(\bar \psi, Q_1,\Psi_1) - \left(\frac{\partial \Psi_1}{\partial x}, Q_1\right) - (g_1,\Psi_1) -(g_2,\psi_1).
\end{align}
We set
$$
\bar\psi ^{'} (y) = -2 \sum _{k=1}^{M} \cos \left( \frac{ 2\pi k}{L} y\right) ,
$$
with $M$ to be chosen later. Then we have:
\begin{align}\label{nonlinearity}
b(\bar \psi, Q_1,\Psi_1) - \left(\frac{\partial \Psi_1}{\partial x} ,Q_1\right) & = \int_{\Omega}  (\bar\psi^{'}-1) \frac{\partial \Psi_1}{\partial x} Q_1 dxdy\notag \\
& = -\int_{\Omega} \left( 2\sum_{k=1}^{M} \cos \left( \frac{2\pi k}{L}y \right) +1 \right) \frac{\partial \Psi_1}{\partial x} Q_1dxdy \notag \\
& = -\int_{\Omega} \left(\sum_{|k|\leq M} e^{\frac{2\pi ik}{L}y}\right) \frac{\partial \Psi_1}{\partial x} Q_1 dxdy \notag \\
& = -\int_{-\frac L2}^{\frac L2} \left[ \sum_{|k|\leq M} \int_{-\frac L 2}^{\frac L 2} e^{\frac{2\pi ik}{L}y} \frac{\partial \Psi_1}{\partial x} Q_1dy \right]dx.
\end{align}
For $(x, y) \in \Omega$, we define
\begin{align}
w(x,y) := \frac{\partial \Psi_1}{\partial x} (x,y)Q_1(x,y);
\end{align}
then
\begin{align}
w(x,y) = \sum_{k \in {\nZ}}w_k(x) e^{-\frac{2\pi ik}{L}y},  \quad \text{with} \quad
w_k(x) = \frac{1}{L} \int_{-\frac L 2}^{\frac L 2} e^{\frac{2\pi ik}{L}y} w(x,y) dy.
\end{align}
Notice that $\frac{\partial \Psi_1}{\partial x}(x,0) = Q_1(x,0) = 0$, for all $x\in[-\frac L2, \frac L 2]$, so $w(x,0) =0$ for all $x\in[-\frac L2, \frac L 2]$. Thus we have for $s>\frac 12$
(see section 4.1, Chapter $3$ in \cite{T}) :
\begin{align}
\left| \sum_{|k|\leq M} w_k(x) \right| &= \left| \sum _{|k|>M} w_k(x)\right|\notag \\
& \leq  \left( \sum_{|k|>M} \left(\frac{2\pi k}{L}\right)^{2s}w_k(x)\right)^{1/2} \left( \sum_{|k|>M} \left (\frac{2\pi k}{L} \right)^{-2s}\right)^{1/2}\notag \\
& \leq C L^{(s-1/2)}M^{(1/2-s)} \|A^{\frac s2}_yw(x,.)\|_{L^2_y},
\end{align}
and since $H^s_{per}(dy)(\Omega)$ is an algebra for $s>1/2$, and since $\frac{\partial \Psi_1}{\partial x}$ and $Q_1$ are odd periodic functions in the variable $y$, we have
\begin{align}
\left| \sum_{|k|\leq M} w_k(x) \right| & \leq C L^{2s-1} M^{(1/2-s)}\left\|\frac{\partial \Psi_1}{\partial x}(x,.)\right\|_{{\dot H}^s_y}\|Q_1(x,.)\|_{{\dot H}^s_y}.
\end{align}
We may choose $s=m-2>1/2$ and conclude from \eqref{nonlinearity}, using the estimates \eqref{Poincare} and \eqref{regularity}, that
\begin{align}
&\left|b(\bar \psi, Q_1,\Psi_1) - \left(\frac{\partial \Psi_1}{\partial x} ,Q_1\right)\right| \notag \\
& \qquad \qquad \leq C L^{2m-5} M^{(5/2-m)}\int_{\frac L2}^{\frac L2}\left\|A^{\frac{(m-2)}{2}}_yQ_1(x,.)\right\|_{L^2_y}\left\|A^{\frac{(m-2)}{2}}_y\frac{\partial \Psi_1}{\partial x}(x,.)\right\|_{L^2_y}\,dx  \notag\\
&\qquad \qquad \leq C L^{2m-5} M^{(5/2-m)}\left\|A^{\frac{(m-2)}{2}}_yQ_1\right\|_{L^2}\left\|A^{\frac{(m-2)}{2}}_y\frac{\partial \Psi_1}{\partial x}\right\|_{L^2}\,dx  \notag\\
& \qquad \qquad \leq CL^{2m-4}(1+L^4)^{1/2}M^{5/2-m}\|A^{m/2}\Psi_1\|_{L^2}^2.
\end{align}
Since $m>5/2$, then we may choose $M = M(L,m)$ large enough such that:
\begin{align}\label{M_two_layer_m}
 CL^{2m-4}(1+L^4)^{1/2}M^{5/2-m} < \frac{\nu}{4},
\end{align}
and thus
\begin{align}
b(\bar \psi, Q_1,\Psi_1) - \left(\frac{\partial \Psi_1}{\partial x} ,Q_1\right)&\leq \frac {\nu}{4}\|A^{m/2}\Psi_1\|_{L^2}^2.
\end{align}
Notice also that
\begin{align}
(g_1,\Psi_1) + (g_2,\psi_2) &= \frac 12\kappa_T(\bar \psi, \Psi) + \nu (A_y^m\bar \psi, \Psi) - \frac 12 \kappa_T(\bar \psi, \psi_2)\notag \\
& = \kappa_T(\bar \psi, \hat \Psi) + \nu(A^{m/2} \bar\psi, A_y^{m/2}\Psi_1)\notag \\
& \leq \frac{\kappa_T}{2}\|\bar \psi\|_{L^2}^2 + \frac{\kappa_T}{2}\|\hat\Psi\|_{L^2}^2 + \nu \|A_y^{m/2}\bar \psi\|_{L^2}^2 + \frac \nu4 \|A^{m/2}\Psi_1\|_{L^2}^2.
\end{align}
As a result, \eqref{E1} can be rewritten as
\begin{align}\label{ee1}
&\frac{d}{dt}\left(\|A^{1/2}\Psi_1\|_{L^2}^2 + \|A^{1/2} \psi_2\|_{L^2}^2 + 2\|\hat\Psi\|_{L^2}^2 \right) + \nu \left(\|A^{m/2} \Psi_1\|_{L^2}^2 + \|A^{m/2}\psi_2\|_{L^2}^2 \right) \notag \\
&\qquad \qquad \qquad \qquad \qquad \qquad + 2\kappa_T \|\hat \Psi\|_{L^2}^2  \leq \kappa_T \|\bar\psi\|_{L^2}^2 + 2\nu\|A_y^{m/2}\bar \psi\|_{L^2}^2.
\end{align}
We set
\begin{align}
E(t)& =  \|A^{1/2} \Psi_1\|_{L^2}^2(t) + \|A^{1/2} \psi_2\|_{L^2}^2(t) + 2\|\hat\Psi\|_{L^2}^2(t),
\end{align}
By \eqref{Poincare} and \eqref{ee1} we conclude that
\begin{align}\label{energy_two_layer_m}
\frac{d}{dt} E(t) + C_3E(t) &\leq \bar\gamma,
\end{align}
where,
\begin{align}
C_3 &= C_3(\nu,L,m,\kappa_T)  = min\left\{\kappa_T, \frac{\nu}{CL^{2(m-1)}}\right\}, \label{C1_two_layer_m} \\
\bar\gamma &= \bar\gamma(\|\bar \psi\|_{L^2}, \|A_y^{m/2}\bar\psi\|_{L^2},\nu,\kappa_T)  = \kappa_T\|\bar\psi\|_{L^2}^2 + 2\nu\|A_y^{m/2}\bar \psi\|_{L^2}^2. \label{bar_gamma_two_layer_m}
\end{align}
This proves the existence of an absorbing ball for system \eqref{two_layer_m} associated with periodic solutions which are odd in the $y$ variable and initial condition $(\psi_1^0, \psi_2^0)\in \nV^1$.

We will now prove the existence of an absorbing ball for system \eqref{two_layer_m} associated with periodic solutions which are odd in the variable $y$ and initial data $(q_1^0,q_2^0)\in \nH$.
From \eqref{E3} and by the interpolation inequality \eqref{Sobolev_interpolation}, the estimate \eqref{Poincare}, and Young's inequality we have
\begin{align}\label{q_estimate_two_layer}
&\frac1 2 \frac{d}{dt}  \left(\sum_{i=1,2}\|q_i\|_{L^2}^2+ \sum_{i=1,2}\|A^{1/2}\psi_i\|_{L^2}^2 + 2\|\hat \psi\|_{L^2}^2\right) \notag \\
&\qquad +  \nu\left(\sum_{i=1,2} \left(\|A^{(m+1)/2} \psi_i\|_{L^2}^2+\|A^{m/2} \psi_i\|_{L^2}^2\right)+ 2\|A^{m/2}\hat \psi\|_{L^2}^2\right) \notag \\
& \qquad \leq \frac{1}{2} \sum_{i=1,2} \left((|\beta|+1)^2\|A^{1/2} \psi_i\|_{L^2}^2 +\|q_i\|_{L^2}^2\right)  +\left|\frac{\kappa_M}{2}-2\kappa_T\right|\|\hat \psi\|_{L^2}^2\notag \\
& \qquad \leq \left((\beta^2+1) + \frac{C(1+L^4)^2L^{2(m-3)}}{\nu} +CL^{2m}\left|\frac{\kappa_M}{2}-2\kappa_T\right|\right)\sum_{i=1,2}\|A^{1/2}\psi_i\|_{L^2}^2 \notag \\
& \qquad \qquad \qquad \qquad + \frac \nu 4 \sum_{i=1,2}\|A^{m/2}\psi_i\|_{L^2}^2.
\end{align}
Set
\begin{align}
W(t) &= \sum_{i=1,2}\|q_i\|_{L^2}^2(t)+  \sum_{i=1,2}\|A^{1/2}\psi_i\|_{L^2}^2(t) + 2\|\hat \psi\|_{L^2}^2(t).
\end{align}
Then, the inequality \eqref{q_estimate_two_layer}  together with the estimates \eqref{Poincare} and \eqref{regularity} imply that
\begin{align}\label{enstrophy_two_layer_m}
&\frac{d}{dt} W(t)+  C_4W(t) \leq  C_5\sum_{i=1,2}\|A^{1/2}\psi_i\|_{L^2}^2,
\end{align}
where
\begin{align}
C_4 &= C_4(\nu,L,m) = min\left\{\frac{\nu}{C(1+L^4)L^{2(m-1)}}, \frac{\nu}{CL^{2m}}\right\}, \label{C2_two_layer_m}\\
C_5 & = C_5(\nu,L,m,\kappa_T,\kappa_M, \beta)\notag \\
& = \left((\beta^2+1) + \frac{C(1+L^4)^2L^{2(m-3)}}{\nu} +CL^{2m}\left|\frac{\kappa_M}{2}-2\kappa_T\right|\right). \label{C4_two_layer_m}
\end{align}
Now we choose $C_6 = C_6(\nu,L,m,\kappa_T) = min\{C_3,C_4\}$: then estimate \eqref{energy_two_layer_m} and estimate \eqref{enstrophy_two_layer_m} yield
\begin{align}
&\frac{d}{dt} E(t) + C_6 E(t) \leq \bar\gamma, \label{energy_two_layer_C} \\
&\frac{d}{dt} W(t)+  C_6 W(t) \leq C_5\sum_{i=1,2}\|A^{1/2}\psi_i\|_{L^2}^2.\label{enstrophy_two_layer_C}
\end{align}
Applying Gronwall's Lemma to \eqref{energy_two_layer_C} implies that:
\begin{align}\label{Es1}
E(t) \leq E(0) e^{-C_6t} + \frac{\bar\gamma}{C_6}(1-e^{-C_6t}).
\end{align}
Since $\psi_1= \bar\psi + \Psi_1$, then we have:
\begin{align}
\|A^{1/2}\psi_1(t)\|_{L^2}^2 \leq 2\|A^{1/2}\Psi_1(t)\|_{L^2}^2 + 2\|A_y^{1/2}\bar\psi\|_{L^2}^2,
\end{align}
for all $t\geq0$. Consequently,

\begin{align}\label{L2_ball_psi_two_layer_m}
\|A^{1/2}\psi_1(t)\|_{L^2}^2+\|A^{1/2}\psi_2(t)\|_{L^2}^2 &\leq 2\|A_y^{1/2}\bar\psi\|_{L^2}^2+ 2E(0) e^{-C_6t} + \frac{2\bar\gamma}{C_6}(1-e^{-C_6t}),
\end{align}
for all $t\geq 0$. From \eqref{enstrophy_two_layer_C} and \eqref{L2_ball_psi_two_layer_m}  we have
\begin{align*}
\frac{d}{dt} W&(t)+  C_6W(t) \leq C_5\left(E(0) e^{-C_6t}+ \frac{2\bar\gamma}{C_6}(1-e^{-C_6t})\right).
\end{align*}
Applying Gronwall's Lemma yields
\begin{align*}
&W(t) \leq  W(0)e^{-C_6t} \leq C_5\left(E(0)te^{-C_6t} + \frac{2\bar\gamma}{C_6^2}(1-e^{-C_6t} -C_6te^{-C_6t})\right),
\end{align*}
for all $t\geq0$. Then we conclude that:
\begin{align*}
&\sum_{i=1,2} \|q_i\|_{L^2}^2(t)  \leq W(0)e^{-C_6t} + C_5\left(E(0)te^{-C_6t} + \frac{2\bar\gamma}{C_6^2}(1-e^{-C_6t} -C_6te^{-C_6t})\right),
\end{align*}
for all $t\geq0$. Consequently, we obtain:
\begin{align}
\limsup_{t\rightarrow \infty} \sum_{i=1,2}\|q_i\|_{L^2}^2(t) \leq \frac{2\bar \gamma}{C_6^2} C_5=:\rho^2.
\end{align}
This proves that the ball centered at $0$ with radius $2\rho$ in $\nH$ is an absorbing ball for system \eqref{two_layer_m} when $(q_1^0, q_2^0) \in \nH$.
%
Moreover,  by the estimates \eqref{Poincare} and \eqref{regularity}, inequality  \eqref{enstrophy_two_layer_C} and inequality \eqref{L2_ball_psi_two_layer_m} imply
\begin{align*}
&\frac{d}{dt} W(t)+ \frac{\nu}{C(1+L^4)L^{2(m-2)}}\sum_{i=1,2}\|A^{1/2}q_i\|_{L^2}^2\notag \\
& \qquad \qquad \qquad  \qquad \qquad  \leq 2C_5\left(\|A^{1/2}\bar\psi\|_{L^2}^2+ E(0) e^{-C_6t} + \frac{\bar\gamma}{C_6}(1-e^{-C_6t})\right).
\end{align*}
After we integrate with respect to time we get:
\begin{align*}
&\frac{\nu}{C(1+L^4)L^{2(m-2)}}\int_0^t  \sum_{i=1,2}\|A^{1/2}q_i\|_{L^2}^2(s)\, ds \leq \notag \\
& \qquad \qquad \qquad \qquad  W(0)+ 2C_5\int_0^t \left(\|A^{1/2}\bar\psi\|_{L^2}^2+ 2E(0) e^{-C_6s} + \frac{\bar\gamma}{C_6}(1-e^{-C_6s})\right)\, ds.
\end{align*}
Thus, we conclude that
\begin{align}\label{two_layer_m_zeta}
&\limsup_{t\rightarrow\infty} \frac{1}{t}\int_0^t \sum_{i=1,2}\|A^{1/2}q_i\|_{L^2}^2(s)\, ds  \notag \\
& \qquad \qquad \qquad \qquad \leq \frac{CC_5(1+L^4)L^{2(m-2)}}{\nu} (\left(\|A^{1/2}\bar\psi\|_{L^2}^2 + \frac{\bar\gamma}{C_6}\right).
\end{align}
Using the general theory of existence of attractors (see for instance \cite{CF88}, and \cite{T}) and the estimates above, we show the existence of a compact global attractor for the dynamical system associated to our model. The existence of an absorbing ball for the system in $L^2_{per}(\Omega)$ together with the fact that the semigroup that generates the flow
$$
S(t):\ (q^0_1,q_2^0)\in \nH\mapsto (q_1(t),q_2(t))\in \nH
$$
is continuous from $\nH$ into $\nH$, for all $t>0$ implies the existence of a global attractor ${\mathcal A}$ for the system in $L^2_{per}(\Omega)$. Since our system is dissipative, i.e. the semigroup $S(t)$ is  continuous from $\nH$ into $\nV^{m-1}$, and since by Rellich Lemma $\nV^{m-1}$ is a compactly embedded in $\nH$, the global attractor ${\mathcal A}$ is compact in $\nH$. Moreover, the global attractor $\mathcal A$ is connected as well. (For a complete discussion of global attractors we refer the readers to \cite{CF88}, \cite{T}, and the references therein).
\begin{theorem} The dynamical system induced by \eqref{two_layer_m} associated with periodic solutions which are odd in the variable $y$ in $\Omega$ and initial data $(q_1^0, q_2^0) \in \nH$ possesses a global attractor ${\mathcal A}$ which is maximal, connected, and compact in $\nH$.
\end{theorem}
Moreover, by a similar argument as above and inequality \eqref{energy_two_layer_m}, which implies the existence of an absorbing ball of system \eqref{two_layer_m} associated with periodic solutions which are odd in the variable $y$ and initial data $(\psi_1^0, \psi_2^0)\in \nV^1$, we conclude the following theorem.
\begin{theorem} The dynamical system induced by \eqref{two_layer_m} associated with periodic solutions which are odd in the variable $y$ in $\Omega$ and initial data $(\psi_1^0, \psi_2^0)\in \nV^1$ possesses a global attractor ${\mathcal A}$ which is maximal, connected, and compact in $\nV^1$.
\end{theorem}
We can also prove the existence of an absorbing set in $\nH^s, \ s=1, 2, 3, \dots,$
using similar methods as those in, for example (\cite{T}), but
we leave the details to the reader.
Our next step is to obtain an upper bound on the Hausdorff and fractal
dimensions of the global attractor $\mathcal A.$ We can also prove the
existence of an absorbing ball in Gevrey space following \cite{FT}
and the generalization presented in \cite{FeTi}. In particular, one can show that the solutions in the attractor are spatially analytic, and have their Fourier coefficients decay exponentially fast.

\bigskip
\section{Upper bound on the dimension of the attractor}
We now turn our attention to estimating the dimension of the global attractor of
the dynamical system associated with \eqref{two_layer_m}.
We start with the study of the linearized equations. System \eqref{two_layer_m}  can be written in the form:
\begin{equation}
\frac{\partial q}{\partial t}+{\bf J}(Rq, q)+L_1q+L_2q-\nu A^m Rq=0 , \quad q(0)= q_0,
\label{abs_two_layer_m}
\end{equation}
where $q=(q_1,q_2)^{T}$, and  $Rq =R(q_1,q_2): =\psi=(\psi_1,\psi_2)^T$, is defined to be the unique solution of the coupled elliptic system
\begin{align}
& q_1 = -A\psi_1 -\hat\psi, \qquad q_2 = -A\psi_2 + \hat\psi, \qquad \hat\psi = \frac{1}{2}(\psi_1-\psi_2),\\
& \int_{\Omega} \psi_1 \,dxdy =0, \quad \int_{\Omega}\psi_2\, dxdy = 0;
\end{align}
and ${\bf J}(\psi, q) = (J(\psi_1,q_1),J(\psi_2,q_2))$, and the linear operators $L_1$ and $L_2$ are defined by:
\begin{align}
L_1q = \left(\frac{\partial q_1}{\partial x} -\kappa_T\hat\psi, -\kappa_MA \psi_2+\kappa_T\hat\psi \right)^{T},\;
L_2q = \left((\beta +\frac 12)\frac{\partial \psi_1}{\partial x}, (\beta -\frac 12)\frac{\partial \psi_2}{\partial x} \right)^{T}.
\end{align}
System \eqref{abs_two_layer_m} can be written in an abstract form as
\begin{equation}
\frac{\partial q}{\partial t} =F(q)q,
\end{equation}
where $F(q)=-{\bf J}(Rq, .) -L_1-L_2 + \nu A^{m} R.$
The first variation equation about $q$, where $q$ is assumed to be the solution of system \eqref{two_layer_m}, is given by (here $Q=(Q_1,Q_2)$)
\begin{align}
\frac{\partial Q}{\partial t}=F'(q)Q,\qquad Q(0)=\xi.
\end{align}
It is easy to check that this equation can be written (formally) as
\begin{align}
\frac{\partial Q}{\partial t}+{\bf J}(Rq, Q)+ {\bf J}(RQ,q)+ {\bf J}(RQ,Q)+L_1Q+L_2Q-\nu A^m RQ=0, \label{linear1}\\
\int_\Omega Q\ dxdy=0,  \qquad Q(0)=\xi.\label{linear3}
\end{align}
Following similar arguments to those in the previous sections, and the Galerkin approximation method, one can show the following;\\
\noindent (i) if $q$ is a solution of \eqref{two_layer_m} with
$q \in L^{\infty}(0,T;\nH)\cap L^2(0,T;\nV^{m-1}), $ for all $T>0$, then for $\xi$ given in $\nH$, equation \eqref{linear1} has a unique solution
$Q$ satisfying
\begin{equation}
Q\in L^{\infty}(0,T;\nH)\cap L^2(0,T;\nV^{m-1});
\end{equation}
\noindent (ii) the semigroup $S(t): \nH \rightarrow \nH$ that satisfies $S(t)q^0 = q$ and $\frac{\partial S(t)q_0}{\partial q_0}\xi = Q(t)$ is Fr\'echet differentiable in $\nH,$
with respect to the initial values.

We now estimate the dimension of the global  attractor $\mathcal A$, following the work of \cite{Constantin_Foias_1985} (see also \cite{CF88} and \cite{T}). Let  $q=q(\tau)=S(\tau)q^0$ be a fixed orbit (i.e. solution of \eqref{abs_two_layer_m}, with time denoted by $\tau$). For $k\in \mathbb N,$ we consider $\xi_1,\dots,\xi_k, \ k$ elements of $\nH,$ and the corresponding solutions $Q_1, \dots, Q_k$ of \eqref{linear1})-\eqref{linear3}, with initial data $Q_j(0)= \xi_j, \, j=1, \dots, k.$
Let $P_k(\tau)=P_k(\tau, q_0;\xi_1,\dots ,\xi_k)$ be the orthogonal projector in $\nH$ onto the space spanned by $\{Q_1(\tau), \dots ,Q_k(\tau)\}.$ At a given time $\tau$, let $\theta_j(\tau),\ j\in \mathbb N,$ be an orthonormal basis, such that $\mbox{Span}\{\theta_1(\tau), \dots, \theta_k(\tau)\}$ = $P_k(\tau)\nH= \mbox{Span}\{Q_1(\tau), \dots , Q_k(\tau)\}.$ Since $Q_j(\tau)\in \nV^{m-1}={\dot H}^{m-1}_{per}\times{\dot H}^{m-1}_{per}$, for a.e. $\tau,$ the functions $\theta_1(\tau), \dots, \theta_k(\tau)$ also belong to $\nV^{m-1}$, for a.e. $\tau.$ We also have
\begin{align}
\mbox{Tr}\left(F'(q(\tau))\circ P_k(\tau)\right)&=\sum_{j=1}^{k}(F'(q(\tau))\circ P_k(\tau)\theta_j(\tau), \theta_j(\tau))\notag \\
&=\sum_{j=1}^{k}(F'(q(\tau))\theta_j(\tau), \theta_j(\tau)).
\end{align}
Writing $\theta_j=(\theta_{j,1},\theta_{j,2}),$ we also define $(\gamma_{j,1}, \gamma_{j,2})= R\theta_j,$ i.e.,
\begin{align*}
\theta_{j,1}=\Delta\gamma_{j,1}-\hat\gamma_j, \quad \theta_{j,2}=\Delta\gamma_{j,2}+\hat\gamma_j, \quad \hat\gamma_j= \frac 12(\gamma_{j,1}-\gamma_{j,2}), \\
\int_\Omega \gamma_{j,i}dxdy=0, \quad \gamma_{j,i} \text{ $\Omega$- periodic, }i=1,2, \ j=1, \dots , k,
\end{align*}
Omitting temporarily the variable $\tau,$ we see that
\begin{align}
(F'(q)\theta_j, \theta_j)=-&\big({\bf J}(Rq, \theta_j),\theta_j\big)-\big({\bf J}(R\theta_j, q),\theta_j\big)- ({\bf J}(R\theta_j,\theta_j),\theta_j)\notag \\
&-(L_1\theta_j, \theta_j)-(L_2\theta_j, \theta_j)+\nu(A^mR\theta_j, \theta_j).
\end{align}
Next, we note that $\big({\bf J}(Rq, \theta_j),\theta_j\big)=({\bf J}(R\theta_j,\theta_j), \theta_j)=0$ and that
\begin{align*}
(A^mR\theta_j, \theta_j)=&(A^m\gamma_{j,1},-A\gamma_{j,1}-\hat \gamma_{j}) +(A^m\gamma_{j,2},-A\gamma_{j,2}+\hat \gamma_{j})\\
=& -\|A^{(m+1)/2}\gamma_j\|_{L^2}^2 -2\|A^{m/2}\hat \gamma_j\|_{L^2}^2,
\end{align*}
where
$$
\|A^{(m+1)/2}\gamma_j\|_{L^2}^2=\|A^{(m+1)/2} \gamma_{j,1}\|_{L^2}^2+\|A^{(m+1)/2}\gamma_{j,2}\|_{L^2}^2.
$$
Also, we have
\begin{align*}
(L_1\theta_j, \theta_j)&= \left( \frac{\partial \theta_{j,1}}{\partial x} -\kappa_T\hat\gamma_j, \theta_{j,1}\right) + \left(- \kappa_MA\gamma_{j,2}+\kappa_T\hat\gamma_j, \theta_{j,2} \right )
\\
&= \left(-\kappa_T\hat\gamma_j, -A \gamma_{j,1}- \hat \gamma_j\right) + \left( -\kappa_MA\gamma_{j,2}+\kappa_T\hat\gamma_j, -A \gamma_{j,2}+\hat \gamma_j \right )\\
&= 2\kappa_T  \|\hat\gamma_j\|_{L^2}^2 + 2\kappa_T\|A^{1/2} \hat \gamma_j\|_{L^2}^2 + \kappa_M\|A\gamma_{j,2}\|_{L^2}^2 - \kappa_M(A \gamma_{j,2},\hat \gamma_j),
\end{align*}
so,
\begin{align}
-(L_1\theta_j,\theta_j) &= -2\kappa_T  \|\hat\gamma_j\|_{L^2}^2 - 2\kappa_T\|A^{1/2} \hat \gamma_j\|_{L^2}^2 - \kappa_M\|A\gamma_{j,2}\|_{L^2}^2 + \kappa_M(A^{1/2} \gamma_{j,2},A^{1/2}\hat \gamma_j) \notag \\
&\leq -\kappa_T \|\hat \gamma_j\|_{L^2}^2 -\kappa_T\|A^{1/2} \hat \gamma_j\|_{L^2}^2 + \left(\frac{\kappa_M^2}{4\kappa_T} - \kappa_M \right) \|A^{1/2} \gamma_{j,2}\|_{L^2}^2.
\end{align}
Furthermore,
\begin{align}
|( L_2\theta_j, \theta_j)|& = \left|\left((\beta+\frac{1}{2})\frac{\partial \gamma_{j,1}}{\partial x} , \theta_{j,1}\right)\right|+ \left|\left((\beta-\frac{1}{2})\frac{\partial \gamma_{j,2}}{\partial x}, \theta_{j,2}\right)\right| \notag\\
& \leq (|\beta|+1) \int_\Omega |A^{1/2}\gamma_j(\tau; x,y)||\theta_j(\tau;x,y)|\, dxdy,
\end{align}
and
\begin{align}
|\big({\bf J}(R\theta_j,q),\theta_j)| &=  \left| \int_\Omega A^{1/2}\gamma_j(\tau;x,y)\theta_j(\tau;x,y)A^{1/2} q(\tau;x,y)\,dxdy\right|\notag \\
&\leq \int_\Omega |A^{1/2}\gamma_j(\tau; x,y)||\theta_j(\tau;x,y)||A^{1/2}q(\tau; x,y)|\, dxdy.
\end{align}
Hence,
\begin{align*}
\sum_{j=1}^{k}(F'(q)\theta_j, \theta_j)&\leq -\nu \sum_{j=1}^{k}
\big[A^{(m+1)/2}\gamma_j|^2+2|A^{m/2}\hat\gamma_j|^2\big]\\
&-\sum_{j=1}^k
\big[\kappa_T \|\hat \gamma_j\|_{L^2}^2 +2\kappa_T\|A^{1/2} \hat \gamma_j\|_{L^2}^2 - \left(\frac{\kappa_M^2}{4\kappa_T} - \kappa_M \right) \|A^{1/2} \gamma_{2,j}\|_{L^2}^2\big]\\
&+\sum_{j=1}^k\int_\Omega |A^{1/2}\gamma_j(\tau; x,y)||\theta_j(\tau;x,y)|\left(|\beta|+1+|A^{1/2}q(\tau; x,y)|\right)\, dxdy.
\end{align*}
Notice that
\begin{align}
\sum_{j=1}^k\int_\Omega |A^{1/2}\gamma_j(\tau; x,y)||\theta_j(\tau;x,y)|\left(|\beta|+1+|A^{1/2}q(\tau; x,y)|\right)\, dxdy \leq \notag \\
\int_{\Omega} \rho^{1/2}(\tau;x,y)\sigma^{1/2}(\tau;x,y)\left(|\beta|+1+|A^{1/2}q(\tau; x,y)|\right)\, dxdy,
\end{align}
where
\begin{align}
\sigma(\tau;x,y) = \sum_{j=1}^k |A^{1/2}\gamma_j(\tau;x,y)|^2, \qquad \rho(\tau;x,y)= \sum_{j=1}^k |\theta_j(\tau;x,y)|^2.
\end{align}
By the H\"older inequality we have
\begin{align}
&\int_{\Omega} \rho^{1/2}\sigma^{1/2}\left(|\beta|+1+|A^{1/2}q(\tau; x,y)|\right)\, dxdy \leq \notag \\
&\qquad \qquad \|\rho\|_{L^2}^{1/2}\|\sigma\|_{L^\infty}^{1/2}\left(\|A^{1/2}q\|_{L^{4/3}} + L^{3/2}(|\beta|+1)\right).
\end{align}
Using a version of the Lieb-Thirring inequality introduced in, e.g., \cite{CF88}, \cite{T}, there is a constant $C$ such that
\begin{align}
\|\rho\|_{L^2}^{1/2}&\leq C \left(\sum_{j=1}^m\|A^{1/2}\theta_j\|_{L^2}^2\right)^{1/4}\leq C(1+L^4)^{1/4}\left(\sum_{j=1}^m\|A^{3/2}\gamma_j\|_{L^2}^2\right)^{1/4},
\end{align}
and by the version we prove in section 8,
\begin{align}
\|\sigma\|_{L^\infty}^{1/2}&\leq CL^{1/2}\left(\sum_{j=1}^m\|A^{3/2}\gamma_j\|_{L^2}^2\right)^{1/4}.
\end{align}
Since
\begin{align*}
\|A^{1/2} q\|_{L^{4/3}}\leq L^{1/2} \|A^{1/2}q\|_{L^2}.
\end{align*}
we have
\begin{align*}
&\int_{\Omega} \rho^{1/2}\sigma^{1/2}\left(|\beta|+1+|A^{1/2}q(\tau; x,y)|\right)\, dxdy \leq \notag \\
& C(1+L^4)^{1/4}L^{1/2}L^{m-2}\left(\sum_{j=1}^m\|A^{(m+1)/2}\gamma_j\|_{L^2}^2\right)^{1/2}\left(L^{1/2} \|A^{1/2}q\|_{L^2}+ L^{3/2}(|\beta|+1)\right).
\end{align*}
By Young's inequality we then have that
\begin{align}
&\int_{\Omega} \rho^{1/2}\sigma^{1/2}\left(|\beta|+1+|A^{1/2}q(\tau; x,y)|\right)\, dxdy \leq \notag \\
& \frac \nu4\sum_{j=1}^m\|A^{(m+1)/2}\gamma_j\|_{L^2}^2 + \frac{C(1+L^4)^{1/2}L^{2m-2}}{\nu} \left(\|A^{1/2}q\|_{L^2}^2 + L^2(|\beta|+1)^2\right).
\end{align}
Using the version of the Lieb-Thirring inequality we prove in section 8, we have
\begin{align}
\left|\frac{\kappa_M^2}{4\kappa_T}-\kappa_M\right| \sum_{j=1}^k\|A^{1/2}\gamma_{j,2}\|_{L^2}^2 & = \int_\Omega \sum_{j=1}^k|A^{1/2}\gamma_j(\tau;x,y)|^2\,dxdy \notag \\
& \leq C_7L\int_\Omega\left(\sum_{j=1}^k\|A^{3/2}\gamma_j\|_{L^2}^2\right)^{1/2}\,dxdy \notag \\
& \leq C_7L^{m+1}\left(\sum_{j=1}^k\|A^{(m+1)/2}\gamma_j\|_{L^2}^2\right)^{1/2} \notag \\
& \leq \frac \nu 4\sum_{j=1}^k\|A^{(m+1)/2}\gamma_j\|_{L^2}^2+ \frac{C_7^2L^{2m+2}}{\nu},
\end{align}
where $C_7 = C_7(\kappa_T,\kappa_M)$.
We then conclude that
\begin{align}
\sum_{j=1}^{k}(F'(q)\theta_j, \theta_j)\leq & -\frac{\nu}{2} \sum_{j=1}^k \|A^{(m+1)/2}\gamma_j\|_{L^2}^2 +  \frac{C_7^2L^{2m+2}}{\nu} \notag \\
& + \frac{C(1+L^4)^{1/2}L^{2m-2}}{\nu} \left(\|A^{1/2}q\|_{L^2}^2 + L^2(|\beta|+1)^2\right).
\end{align}
Now we note that, since $\theta_j, \ j=1, \dots ,k$ are orthonormal in $\nH$, the
eigenvalues $(\lambda_j)_{j\in \mathbb N}$ of $A^{m-1}$ satisfy
\begin{equation}
 \lambda _j \sim C L^{-2(m-1)}j^{m-1}, \quad\quad \text{for all }j \in \mathbb N,
\end{equation}
we have (see, for example, Lemma 2.1, Chapter VI in \cite{T})
\begin{equation}
\sum_{j=1}^{k}
|A^{(m-1)/2} \theta_j|^2\geq CL^{-2(m-1)}k^m,
\end{equation}
for some positive absolute constant $C$, and therefore from the elliptic regularity estimate \eqref{regularity} we conclude that
\begin{align}
\sum_{j=1}^{k}(F'(q)\theta_j, \theta_j)\leq &\frac{-\nu C}{(1+L^4)L^{2(m-1)}}k^m + \frac{C_7^2L^{2m+2}}{\nu} \notag \\
& + \frac{C(1+L^4)^{1/2}L^{2m-2}}{\nu} \left(\|A^{1/2}q\|_{L^2}^2 + L^2(|\beta|+1)^2\right).
\end{align}
From this we obtain
\begin{align}\label{two_layer_m_trace_limsup}
&\limsup_{t\to \infty}\frac 1t\int_0^t\sum_{j=1}^{k}(F'(q(\tau)\theta_j(\tau), \theta_j(\tau))\,d\tau \leq \frac{-\nu C}{(1+L^4)L^{2(m-1)}}k^m + \frac{C_7^2L^{2m+2}}{\nu} \notag \\
& \qquad \qquad \qquad \qquad \qquad + \frac{C(1+L^4)^{1/2}L^{2m-2}}{\nu} \left(\zeta + L^2(|\beta|+1)^2\right),
\end{align}
where from the previous section, \eqref{two_layer_m_zeta} implies that
\begin{align}\label{zeta_two_layer_m}
\zeta&:= \limsup_{t\to\infty} \frac 1t\int_0^t \|A^{1/2} q(\tau)\|_{L^2}^2\, d\tau \notag \\
& \qquad \qquad \qquad \qquad \leq \frac{CC_5(1+L^4)L^{2(m-2)}}{\nu} (\left(\|A^{1/2}\bar\psi\|_{L^2}^2 + \frac{\bar\gamma}{C_6}\right).
\end{align}
where ${\bar\gamma}$ is given in \eqref{bar_gamma_two_layer_m}, $C_6=\min\{C_3,C_4\}$, where $C_3$ and $C_4$ are given in \eqref{C1_two_layer_m} and {\eqref{C2_two_layer_m}, respectively, and $C_5$ is given in \eqref{C4_two_layer_m}.

Now we assume that $q_0$ belongs to the global attractor $\mathcal A$
and introduce the quantities $r_k(t)$ and $R_k,$
\begin{equation}
r_k(t)=\sup_{q_0\in \mathcal A} \sup_
{|\xi_i|\leq 1,i = 1,\dots k}
\big(\frac 1t\int_0^t\mbox{Tr }F'(S(\tau)q_0)\circ P_k(\tau)\,d\tau\big),
\end{equation}
\begin{equation}
R_k=\limsup_{t\to \infty}r_k(t).
\end{equation}
From \eqref{two_layer_m_trace_limsup}, we have that
\begin{align*}
R_k \leq  \frac{-\nu C}{(1+L^4)L^{2(m-1)}}k^m + \frac{C_7^2L^{2m+2}}{\nu} + \frac{C(1+L^4)^{1/2}L^{2m-2}}{\nu} \left(\zeta+ L^2(|\beta|+1)^2\right).
\end{align*}
With the definition
\begin{align}
d := min \left\{k \in \nN: R_k \leq 0 \right\},
\end{align}
one can easily prove that
\begin{align}
d-1 < \left(\frac{(1+L^4)C_7^2L^{4m}}{C\nu^2} + \frac{C(1+L^4)^{3/2}L^{4m-4}}{\nu^2} \left(\zeta+ L^2(|\beta|+1)^2\right)\right)^{1/m} \leq d.
\label{dimension_two_layer_m}
\end{align}
Then, from the general theory for the dimension of the attractors
(see for instance \cite{CI}, \cite{CF88} and \cite{T}),
we obtain that the Hausdorff dimension of the attractor $\mathcal A$ is less
than or equal to $d.$ We have just proved the following:
\begin{theorem}  Consider the dynamical system associated with the two-layer model \eqref{two_layer_m} restricted to periodic functions which are odd in the variable $y$ with periodic boundary conditions. We denote by $\mathcal A$ its global attractor, and we let $d$ be defined to satisfy \eqref{dimension_two_layer_m} where $\zeta$, $\bar \gamma$, $C_3$, $C_4$ and $C_5$ are given in \eqref{zeta_two_layer_m}, \eqref{bar_gamma_two_layer_m}, \eqref{C1_two_layer_m}, \eqref{C2_two_layer_m} and \eqref{C4_two_layer_m}, respectively. $C_6 = \min\{C_3, C_4\}$ and $C_7 = C_7(\kappa_T,\kappa_M)$.
\begin{align*}
\bar\psi ^{'} (y)&= -2 \sum _{k=1}^{M} \cos \left( \frac{ 2\pi k}{L} y\right) ,
\end{align*}
with $M$ is chosen to be large enough to satisfy \eqref{M_two_layer_m}. Then the Hausdorff dimension of $\mathcal A$ is less than or equal to $d$ and its fractal dimension is less than or equal to $2d$.
\end{theorem}

\bigskip
\section{Existence of an inertial manifold}
In this section we sketch the proof of existence of an inertial manifold for the two-layer quasi-geostrophic model \eqref{two_layer_m} when $m>5/2$.
We start by recalling the definition of an inertial manifold. Let $H$ be a real Hilbert space and consider the differential equation
\begin{equation}
\frac{\partial q}{\partial t}+ Mq +G(q)=0, \quad q(0)=q_0\in H. \label{I.1}
\end{equation}
Here $M$ is a linear self-adjoint, positive definite, densely defined operator with compact resolvant, while $G$ is a nonlinear operator defined on $D(M),$
the domain of the operator $M$.
An inertial manifold for \eqref{I.1} is a Lipschitz manifold $\mathcal M$ that enjoys the following properties:(i) $\mathcal M$ is finite dimensional,
(ii) $\mathcal M$ is positively invariant under the semi-flow generated by \eqref{I.1},
(iii) $\mathcal M$ is a globally attracting set for the solutions of \eqref{I.1} at an exponential rate. For more details we refer the reader to the general references to the theory, and different ways of establishing inertial manifolds, \cite{Constantin_Foias_Nicoleanko_Temam_1989} and \cite{Constantin_Foias_Nicoleanko_Temam_1989_book}. (see also \cite{FSTi} and \cite{T}).

For the two-layer quasi-geostrophic model \eqref{two_layer_m}, we have, $H = \nH = \dot{L}^2_{per} \times  \dot{L}^2_{per}$, $q=(q_1,q_2)^T \in \nH$, and
\begin{align*}
M := A^{m}R, \qquad G(q): =L_1q+L_2q +{\bf J}(Rq,q),
\end{align*}
where $L_1$ and $L_2$ are linear operators and are defined by
\begin{align*}
L_1q= \left(\frac{\partial q_1}{\partial x} -\kappa_T\hat\psi, -\kappa_MA \psi_2+\kappa_T\hat\psi \right ) ^{T},\quad
L_2q= \left((\beta +\frac 12)\frac{\partial \psi_1}{\partial x}, (\beta -\frac 12)\frac{\partial \psi_2}{\partial x} \right)^{T}.
\end{align*}
Recall that, due to the elliptic regularity estimate,
\begin{align}
|Mq| \sim | A^{m-1}q|
\end{align}
and thus, it is fairly easy to check that for any $m>5/2$ we have
\begin{align}
&|M^{-1/2}G(q)-M^{-1/2}G(\bar q)|\leq C(K, L,m) |Mq-M\bar q|,
\end{align}
for all $q$, ${\bar q}$ satisfying $|Mq|\leq K, |M\bar q|\leq K$,and that $G$ is a bounded mapping from $D(M)$ into $D(M^{-1/2}).$
Furthermore, if $\lambda= \lambda_N$ and $\Lambda=\lambda_{N+1},$ we have $\lambda_N \sim c_0\lambda_1N^{m-1}$ as $N\to \infty.$ The spectral gap condition (see for instance \cite{FSTi} or \cite{T}) is satisfied and the existence of an inertial manifold follows from the general theory (for example, see \cite{FSTi} and Theorem 3.1 in \cite{T}).
\begin{theorem} Assume that $m>5/2$, then the dynamical system associated with the two-layer quasi-geostrophic model \eqref{two_layer_m} restricted to periodic functions which are odd in the variable $y$ with periodic boundary conditions posses an inertial manifold.
\end{theorem}
\bigskip
\section{Appendix: A Lieb-Thirring type inequality}
Using a technique due to C. Foias, we will prove the following inequality (see also \cite{Foias_Holm_Titi_2002} and \cite{ITi}).
\begin{lemma} There exist an absolute positive constant $C$, such that for every family of functions $\theta_1, \dots, \theta_k$ in  $\nV^1$, which is orthonormal in $\nH$ we have
\begin{align}
\left\|\sum_{j=1}^k|A^{1/2}\gamma_j(x,y)|^2\right\|_{L^{\infty}} \leq CL \left(\sum_{j=1}^k \|A^{3/2}\gamma_j\|_{L^2}^2\right)^{1/2}
\end{align}
where $\gamma_j= (\gamma_{j,1},\gamma_{j,2}), \ j=1,\dots ,k$, is the unique solution of
\begin{align}
& \theta_{j,i}=\Delta \gamma_{j,i}+ \frac{(-1)^i}{2}(\gamma_{j,1}-\gamma_{j,2}),\quad \int_\Omega \gamma_{j,i}\,dxdy=0, \quad i=1,2.
\end{align}
\end{lemma}
\begin{proof}
Applying the Agmon's inequality in space dimension two \eqref{Agmon} we obtain
\begin{align}
\|A^{1/2} \gamma_j\|_{L^\infty}&\leq C\|A^{1/2} \gamma_j\|_{L^2}^{1/2}\|A^{3/2}\gamma_j\|_{L^2}^{1/2},
\end{align}
which implies by the elliptic regularity estimate \eqref{regularity} and by the orthonormality of $\theta_1, \dots, \theta_k$ that:
\begin{align}\label{A_half_Lieb_Thiring_two_layer}
\|A^{1/2} \gamma_j\|_{L^\infty}&\leq CL^{1/2}\|\theta_j\|_{L^2}^{1/2}\|A^{3/2}\gamma_j\|_{L^2}^{1/2}= CL^{1/2}\|A^{3/2}\gamma_j\|_{L^2}^{1/2}.
\end{align}
Suppose that $\alpha_1, \dots, \alpha_k \in \nR$ such that $\sum_{j=1}^k \alpha_j^2 = 1$. Using inequality \eqref{A_half_Lieb_Thiring_two_layer} we show that

\begin{align}
\left|\sum_{j=1}^{k} \alpha_jA^{1/2}\gamma_j(x,y)\right| &\leq CL^{1/2}\left\|\sum_{j=1}^{k} \alpha_j A^{3/2}\gamma_j\right\|_{L^2}^{1/2} \notag \\
& \leq C L^{1/2} \left( \int_\Omega \left(\sum_{j=1}^k |A^{3/2}\gamma_j(x,y)|^2\right)\,dxdy\right)^{1/4} \notag \\
& = C L^{1/2} \left(\sum_{j=1}^k \left(\int_\Omega |A^{3/2}\gamma_j(x,y)|^2\,dxdy\right)\right)^{1/4},
\end{align}
thus
\begin{align}
\left|\sum_{j=1}^k \alpha_j A^{1/2}\gamma_j (x,y)\right| \leq CL^{1/2} \left(\sum_{j=1}^k \|A^{3/2}\gamma_j\|_{L^2}^2\right)^{1/4}.
\end{align}
Recall that $A^{1/2}\gamma_j = A^{1/2}\gamma_{j,1} {\bf e}_1 + A^{1/2}\gamma_{j,2} {\bf e}_2$, then the above inequality implies
\begin{align*}
\left( \sum_{j=1}^k \alpha_j A^{1/2}\gamma_{j,1} (x,y)\right)^2 + \left( \sum_{j=1}^k \alpha_j A^{1/2}\gamma_{j,2} (x,y)\right)^2 \leq CL  \left(\sum_{j=1}^k \|A^{3/2}\gamma_j\|_{L^2}^2\right)^{1/2},
\end{align*}
for every $\alpha_1, \dots, \alpha_k \in \nR$ that satisfy $\sum_{j=1}^k \alpha_j^2 = 1$. \\
First, we choose $\alpha_j = A^{1/2}\gamma_{j,1}(x,y) / \left(\sum_{j=1}^k \left(A^{1/2}\gamma_{j,1}(x,y)\right)^2\right)^{1/2}$, and later we choose $\alpha_j = A^{1/2}\gamma_{j,2}(x,y) / \left(\sum_{j=1}^k \left(A^{1/2}\gamma_{j,2}(x,y)\right)^2\right)^{1/2}$ to get that
\begin{align*}
\sum_{j=1}^k|A^{1/2}\gamma_j(x,y)|^2 \leq CL \left(\sum_{j=1}^k \|A^{3/2}\gamma_j\|_{L^2}^2\right)^{1/2},
\end{align*}
and thus
\begin{align}
\left\|\sum_{j=1}^k|A^{1/2}\gamma_j(x,y)|^2\right\|_{L^{\infty}} \leq CL \left(\sum_{j=1}^k \|A^{3/2}\gamma_j\|_{L^2}^2\right)^{1/2}.
\end{align}
\end{proof}
\section*{Acknowledgements}
This paper is dedicated to Professor Peter Constantin,
on the occasion of his  60th birthday, as token of friendship
and admiration for his contributions to research in partial
differential equations and fluid mechanics. We would also like to thank Professor Ciprian Foias for the stimulating and inspiring discussions regarding this work.
The work of A.F. and E.S.T. was supported in part by the NSF grants DMS-1009950,
DMS-1109640 and DMS-1109645. E.S.T.  also acknowledges the support of the Alexander von Humboldt
Stiftung/Foundation and the Minerva Stiftung/Foundation. {The work of  M.Z. was partially supported by the NSF grant DMS-1109562.}



\end{document}